\documentclass{amsart}%
\usepackage{amssymb}
\usepackage{amsfonts}
\usepackage{amsmath}
\usepackage{graphicx}%
\newtheorem{theorem}{Theorem}
\theoremstyle{plain}

\newtheorem{corollary}{Corollary}

\newtheorem{definition}{Definition}
\newtheorem{example}{Example}

\newtheorem{lemma}{Lemma}

\newtheorem{proposition}{Proposition}
\newtheorem{remark}{Remark}

\numberwithin{equation}{section}
\begin{document}
\title[Archimedean zeta functions for analytic mappings ]{Poles of Archimedean zeta functions for analytic mappings}
\author{E. Le\'{o}n-Cardenal}
\address{Centro de Investigaci\'{o}n y de Estudios Avanzados del Instituto
Polit\'{e}cnico Nacional\\
Departamento de Matem\'{a}ticas- Unidad Quer\'{e}taro\\
Libramiento Norponiente \#2000, Fracc. Real de Juriquilla. Santiago de
Quer\'{e}taro, Qro. 76230\\
M\'{e}xico.}
\email{eleon@math.cinvestav.mx}
\author{Willem Veys}
\address{University of Leuven, Department of Mathematics\\
Celestijnenlaan 200 B, B-3001 Leuven (Heverlee), Belgium}
\email{wim.veys@wis.kuleuven.be}
\author{W. A. Z\'{u}\~{n}iga-Galindo}
\address{Centro de Investigaci\'{o}n y de Estudios Avanzados del Instituto
Polit\'{e}cnico Nacional\\
Departamento de Matem\'{a}ticas- Unidad Quer\'{e}taro\\
Libramiento Norponiente \#2000, Fracc. Real de Juriquilla. Santiago de
Quer\'{e}taro, Qro. 76230\\
M\'{e}xico.}
\email{wazuniga@math.cinvestav.edu.mx}
\thanks{The third author was partially supported by Conacyt, Grant \# 127794.}

\begin{abstract}
Let $\boldsymbol{f}=$ $\left(  f_{1},\ldots,f_{l}\right)  :U\rightarrow K^{l}%
$, with $K=\mathbb{R}$ or $\mathbb{C}$, be a $K$-analytic mapping defined on
an open set $U$ $\subset K^{n}$, and let $\Phi$ be a smooth function on $U$
with compact support. In this paper, we give a description of the possible
poles of the local zeta function attached to $\left(  \boldsymbol{f}%
,\Phi\right)  $ in terms of a log-principalization of the ideal $\mathcal{I}%
_{\boldsymbol{f}}=(f_{1},\ldots,f_{l})$. When $\boldsymbol{f}$ is a
non-degenerate mapping, we give an explicit list for the possible poles of
$Z_{\Phi}(s,\boldsymbol{f})$ in terms of the normal vectors to the supporting
hyperplanes of a Newton polyhedron attached to $\boldsymbol{f}$, and some
additional vectors (or rays) that appear in the construction of a simplicial
conical subdivision of the first orthant. These results extend the
corresponding results of Varchenko to the case $l\geq1$, and $K=\mathbb{R}$ or
$\mathbb{C}$. In the case $l=1$ and $K=\mathbb{R}$, Denef and Sargos proved
that the candidates poles induced by the extra rays required in the
construction of a simplicial conical subdivision can be discarded from the
list of candidate poles. We extend the Denef-Sargos result arbitrary $l\geq1$.
This
%is a strong and interesting result, yielding
yields in general a much shorter list of candidate poles, that can moreover be
read off immediately from $\Gamma\left(  \boldsymbol{f}\right)  $.

\end{abstract}
\maketitle

\section{Introduction}

We take $K=\mathbb{R}$ or $\mathbb{C}$. Let $\boldsymbol{f}=$ $\left(
f_{1},\ldots,f_{l}\right)  :U\rightarrow K^{l}$ be a $K$-analytic mapping
defined on an open $U$ in $K^{n}$. Let $\Phi:U\rightarrow\mathbb{C}$ be a
smooth function on $U$ with compact support. Then the local zeta function
attached to $\left(  \boldsymbol{f},\Phi\right)  $ is defined as
\[
Z_{\Phi}(s,\boldsymbol{f})=\int\limits_{K^{n}\smallsetminus\boldsymbol{f}%
^{-1}\left(  0\right)  }\Phi\left(  x\right)  \left\vert \boldsymbol{f}%
(x)\right\vert _{K}^{s}|dx|,
\]
for $s\in\mathbb{C}$ with $\operatorname{Re}(s)>0$, where $|dx|$ is the Haar
measure on $K^{n}$. The local zeta functions have a meromorphic continuation
to the whole complex plane. In the case $l=1$, the meromorphic continuation of
$Z_{\Phi}(s,\boldsymbol{f})$ was established jointly by Bernstein and Gel'fand
\cite{B-G}, independently by Atiyah \cite{At}, then by a different method by
Bernstein \cite{Ber}. In \cite{Igusa1}, see also \cite{Igusa3}, Igusa
developed a uniform theory for local zeta functions over local fields of
characteristic zero. In this context, there exist asymptotic expansions for
oscillatory integrals depending on one parameter which are controlled by the
poles of `twisted versions' of $Z_{\Phi}(s,\boldsymbol{f})$, see also
\cite{AVG}, \cite{Var}. In \cite{P-St} , with $l\geq1$ and $K=$ $\mathbb{C}$,
Phong and Sturm studied the stability of the poles of $Z_{\Phi}%
(s,\boldsymbol{f})$ under small perturbations of $\boldsymbol{f}$.

In this paper, we give a geometric description of the possible poles of
$Z_{\Phi}(s,\boldsymbol{f})$, including the largest one, in terms of a
log-principalization of the ideal $\mathcal{I}_{\boldsymbol{f}}=(f_{1}%
,\ldots,f_{l})$, see Theorem \ref{Th1} and Proposition \ref{prop1}. To the
best of our knowledge a such description has not been reported before, in the
Archimedean context. (See however Remark \ref{real case} and Proposition
\ref{princ-resol} when $K=\mathbb{R}$.) When $\boldsymbol{f}$ is a
non-degenerate mapping in the sense of \cite{V-Z}, we give an explicit list
for the possible poles of $Z_{\Phi}(s,\boldsymbol{f})$ in terms of the normal
vectors to the supporting hyperplanes of a Newton polyhedron attached to
$\boldsymbol{f}$, and additional vectors (or rays) that appear in the
construction of a simplicial conical subdivision of the first orthant, see
Theorem \ref{Theorem2a}. These results extend the corresponding results of
Varchenko in \cite{Var} for $l=1$ and $K=\mathbb{R}$ to the case $l\geq1$, and
$K=\mathbb{R}$ or $\mathbb{C}$. In the case $l=1$ and $K=\mathbb{R}$, Denef
and Sargos proved in \cite{D-S} a strong and interesting result: the candidate
poles induced by the extra rays required in the construction of a simplicial
conical subdivision can be discarded from the list of candidate poles in
Theorem \ref{Theorem2a}. We extend the Denef-Sargos result to arbitrary
$l\geq1$, see Theorems \ref{Theorem3}, \ref{Theorem4}, and \ref{Theorem5}.
This yields in general a much shorter list of candidate poles, that can
moreover be read off immediately from $\Gamma\left(  \boldsymbol{f}\right)  $.

At this point we must mention that in the case $K=\mathbb{R}$ a description of
the poles of $Z_{\Phi}(s,\boldsymbol{f})$ can be obtained by using an embedded
resolution of singularities of $\sum_{i=1}^{l}f_{i}^{2}$, see \cite{AVG},
\cite{Igusa3}, \cite{Var}. In particular, one could define $\boldsymbol{f}$ to
be non-degenerate as meaning $\sum_{i=1}^{l}f_{i}^{2}$ to be non-degenerate in
the usual sense, and thus one could use all the results of \cite{Var}. We note
however that there exist many mappings $\boldsymbol{f} $ which are
non-degenerate in the sense of our Definition \ref{def1} but such that
$\sum_{i=1}^{l}f_{i}^{2}$ is degenerate in the usual sense, see Remark
\ref{note}. Thus, our approach gives a finer explicit description of the poles
of $Z_{\Phi}(s,\boldsymbol{f})$.

%WIM: I would suggest to delete the next sentence because there is no DS-result yet concerning the complex numbers.

%By combining the results presented here and
%those of \cite{V-Z}, one obtains a geometric description of the poles of local
%zeta functions attached to analytic mappings, in Archimedean and
%non-Archimedean settings, in terms of log-principalization of ideals.

\section{Local Zeta Functions for Analytic Mappings}

\subsection{\label{fixdata}Fixing the data}

We take $K=\mathbb{R}$ or $\mathbb{C}$. For $a=\left(  a_{1},\ldots
,a_{n}\right)  \in K^{n}$ we put \ $\left\vert a\right\vert _{K}=$ $\left\vert
a\right\vert _{\mathbb{R}}$ or $\left\vert a\right\vert _{\mathbb{C}}^{2}$,
where $\left\vert \cdot\right\vert _{\mathbb{R}}$ and $\left\vert
\cdot\right\vert _{\mathbb{C}}$ are the standard norms on $\mathbb{R}^{n}$ and
$\mathbb{C}^{n}$, respectively.

Let $f_{1},\ldots,f_{l}$ be polynomials in $K\left[  x_{1},\ldots
,x_{n}\right]  $, or, more generally, $K-$analytic functions on an open set
$U\subset K^{n}$. We consider the mapping $\boldsymbol{f}=$ $\left(
f_{1},\ldots,f_{l}\right)  :K^{n}\rightarrow K^{l}$, respectively,
$U\rightarrow K^{l}$. Let $\Phi:K^{n}\rightarrow\mathbb{C}$ \ be a smooth
function with compact support, i.e. $\Phi\in C_{0}^{\infty}$, with support in
$U$ in the second case.

\subsection{Log-principalization of ideals}

We state the version of log-principalization of ideals that we will use in
this paper, \cite{E-N-V}, see also \cite{H}, \cite{W}.

\begin{theorem}
[\cite{E-N-V}]\label{theoem1}Let $K=\mathbb{R}$ or $\mathbb{C}$ and let $U$ be
an open submanifold of $K^{n}$. Let $f_{1},\ldots,f_{l}$ be $K-$analytic
functions on $U$ such that the ideal $\mathcal{I}_{\boldsymbol{f}}=\left(
f_{1},\ldots,f_{l}\right)  $ is not trivial. Then there exists a
log-principalization $h:X_{K}\rightarrow U$ of $\mathcal{I}_{\boldsymbol{f}}$,
that is,

\noindent(1) $X_{K}$ is an $n-$dimensional $K-$analytic manifold, $h$ is a
proper $K-$analytic map which can be chosen as a composition of a finite
number of blow-ups in closed submanifolds, and which is an isomorphism outside
of the common zero set $Z_{K}$ of $f_{1},\ldots,f_{l}$;

\noindent(2) $h^{-1}\left( Z_{K}\right)  =\cup_{i\in T}E_{i}$, where the
$E_{i}$\ are closed submanifolds of $X_{K}$ of codimension one, each equipped
with a pair of positive integers $\left(  N_{i},v_{i}\right)  $ satisfying the
following. At every point $b$ of $X_{K}$ there exist local coordinates
$\left(  y_{1},\ldots,y_{n}\right)  $ on $X_{K}$ around $b$ such that, if
$E_{1},\ldots,E_{r}$ are the $E_{i}$ containing $b$, we have on some
neighborhood of $b$ that $E_{i}$ is given by $y_{i}=0$ for $i=1,\ldots,r$,
\[
h^{\ast}\left(  \mathcal{I}_{\boldsymbol{f}}\right)  \text{ is generated by
}\varepsilon\left(  y\right)
%TCIMACRO{\dprod \limits_{i=1}^{r}}%
%BeginExpansion
{\displaystyle\prod\limits_{i=1}^{r}}
%EndExpansion
y_{i}^{N_{i}},
\]
and
\[
h^{\ast}\left(  dx_{1}\wedge\ldots\wedge dx_{n}\right)  =\eta\left(  y\right)
\left(
%TCIMACRO{\dprod \limits_{i=1}^{r}}%
%BeginExpansion
{\displaystyle\prod\limits_{i=1}^{r}}
%EndExpansion
y_{i}^{v_{i}-1}\right)  dy_{1}\wedge\ldots\wedge dy_{n},
\]
where \ $\varepsilon\left(  y\right)  $, $\eta\left(  y\right)  $ are units in
the local ring of $X_{K}$\ at $b$.
\end{theorem}

The $\left(  N_{i},v_{i}\right)  $, $i\in T$, are called \textit{the numerical
data of }$h$\textit{\ for} $\mathcal{I}_{\boldsymbol{f}}$.

\subsection{Poles of local zeta functions}

From now on we suppose that $\boldsymbol{f}^{-1}\left(  0\right)
\neq\emptyset$. To $\boldsymbol{f}$ and $\Phi$ as in \ref{fixdata} we
associate the local zeta function $Z_{\Phi}(s,\boldsymbol{f})$, $s\in
\mathbb{C}$ with $\operatorname{Re}(s)>0$, defined in the introduction.

\begin{remark}
\textrm{\label{real case} (1) When $K=\mathbb{R}$ the zeta function $Z_{\Phi
}(s,f)$ of the mapping $f$ is clearly equal to the zeta function $Z_{\Phi
}(\frac{s}{2},F)$ of the function $F:=\sum_{i=1}^{l}f_{i}^{2}$. In particular,
it is known that $Z_{\Phi}(s,f)$ has a meromorphic continuation to the whole
complex plane, and that its poles are negative rational numbers and of order
at most $n$, see e.g. \cite{Igusa5}. Also the list of candidate poles in
Theorem \ref{Th1} below for $K=\mathbb{R}$ can in fact be derived from the
function case, see Proposition \ref{princ-resol}. But since the proof of
Theorem \ref{Th1} for $K=\mathbb{C}$ is also valid for $K=\mathbb{R}$ we
prefer to state and prove it simultaneously for both fields. }

\textrm{(2) Over $K=\mathbb{C}$ the meromorphic continuation can be
analogously reduced to the case of one real--analytic function. The point of
Theorem \ref{Th1} is the description of the candidate poles in terms of a
principalization of the ideal. }
\end{remark}

\begin{theorem}
\label{Th1}Let $\boldsymbol{f}$ and $\Phi$\ be as in \ref{fixdata}. Let $h:$
$X_{K}\rightarrow U$ be a fixed log-principa\-lization of the ideal
$\mathcal{I}_{\boldsymbol{f}}=(f_{1},\ldots,f_{l})$, with numerical data
$\left(  N_{i},v_{i}\right)  $, $i\in T$, for $\mathcal{I}_{\boldsymbol{f}}$.
Then $Z_{\Phi}\left(  s,\boldsymbol{f}\right)  $ has a meromorphic
continuation to the whole complex plane $\mathbb{C}$ and the poles are
contained in the union of%
\[
-\frac{v_{i}}{N_{i}}-\frac{\mathbb{N}}{N_{i}},\text{ }i\in T.
\]
Therefore the poles are negative rational numbers. Moreover their orders are
at most equal to $n$.
\end{theorem}

\begin{proof}
We use all the notations concerning the log-principalization $h$ introduced in
Theorem \ref{theoem1}. Let $b\in X_{K}$ be a point, and $\left(  \phi
_{V},V\right)  $ a chart containing it. Let $E_{1},\dots,E_{r}$ denote the
components of $h^{-1}\left(  \boldsymbol{f}^{-1}\left(  0\right)  \right)  $
passing through $b$. We set $\boldsymbol{f}^{\ast}\left(  y\right)
:=\boldsymbol{f}\left(  h\left(  y\right)  \right)  $.

If $r=0$, i.e., $\boldsymbol{f}\left(  h\left(  b\right)  \right)  \neq0$,
then we can choose a small neighborhood $V_{b}$ of $b$ over which $\left\vert
\boldsymbol{f}^{\ast}\left(  y\right)  \right\vert _{K}$ is positive and
$\mathbb{R}$-analytic, and thus
\begin{equation}
\left\vert \boldsymbol{f}^{\ast}\left(  y\right)  \right\vert _{K}^{s}%
=e^{s\ln\left\vert \boldsymbol{f}^{\ast}\left(  y\right)  \right\vert
_{K}\text{ }}\text{is }\mathbb{R}\text{-analytic in }y\in V_{b}\text{ and
holomorphic in }s\in\mathbb{C}. \label{cond-1}%
\end{equation}
In addition,
\begin{equation}
h^{\ast}\left(  dx_{1}\wedge\ldots\wedge dx_{n}\right)  =\eta(y)\left(
dy_{1}\wedge\ldots\wedge dy_{n}\right)  , \label{cond0}%
\end{equation}
where $\eta(y)$ is a unit of the local ring of $X_{K}$ at $b$.

If $r\geq1$, then we have in $V$ that%
\begin{equation}
f_{i}^{\ast}(y)=f_{i}\left(  h\left(  y\right)  \right)  =g(y)\widetilde
{f}_{i}(y)\text{, }i=1,\ldots,l\text{,} \label{cond1}%
\end{equation}%
\begin{equation}
g(y)=\varepsilon\left(  y\right)
%TCIMACRO{\tprod \limits_{i=1}^{r}}%
%BeginExpansion
{\textstyle\prod\limits_{i=1}^{r}}
%EndExpansion
y_{i}^{N_{i}}, \label{cond2}%
\end{equation}
and
\begin{equation}
h^{\ast}\left(  dx_{1}\wedge\ldots\wedge dx_{n}\right)  =\eta(y)%
%TCIMACRO{\tprod \limits_{i=1}^{r}}%
%BeginExpansion
{\textstyle\prod\limits_{i=1}^{r}}
%EndExpansion
y_{i}^{v_{i}-1}\left(  dy_{1}\wedge\ldots\wedge dy_{n}\right)  , \label{cond3}%
\end{equation}
where $\varepsilon\left(  y\right)  $ and $\eta(y)$ are units of the local
ring of $X_{K}$ at $b$. Furthermore, there exists an index $i_{0}$ such that
$\widetilde{f}_{i_{0}}(b)\neq0$. Then
\[
\left\vert \boldsymbol{f}^{\ast}(y)\right\vert _{K}^{s}= \left\vert
\varepsilon\left(  y\right)  \right\vert _{K}^{s} \left(
%TCIMACRO{\tprod \limits_{i=1}^{r}}%
%BeginExpansion
{\textstyle\prod\limits_{i=1}^{r}}
%EndExpansion
\left\vert y_{i}\right\vert _{K}^{N_{i}}\right)  ^{s} \left\vert
\widetilde{\boldsymbol{f}}(y)\right\vert _{K}^{s},
\]
where $\widetilde{\boldsymbol{f}}(y):=\left(  \widetilde{f}_{1}(y),\ldots
,\widetilde{f}_{l}(y)\right)  $, in $V$.

We can choose a small neighborhood $V_{b}$ of $b$ over which (\ref{cond1}%
)-(\ref{cond3}) \ are valid, and $\left\vert \varepsilon\left(  y\right)
\right\vert _{K}$, $\left\vert \widetilde{\boldsymbol{f}}(y)\right\vert _{K}$,
$\left\vert \eta\left(  y\right)  \right\vert _{K}$ are $\mathbb{R}$-analytic.
Then $\left\vert \varepsilon\left(  y\right)  \right\vert _{K}^{s}$,
$\left\vert \widetilde{\boldsymbol{f}}(y)\right\vert _{K}^{s}$ are
$\mathbb{R}$-analytic in $y$ for any $s\in\mathbb{C}$, and holomorphic in
$s\in\mathbb{C}$, for any $y\in V_{b}$.

Since $h^{-1}\left(  \text{support }\Phi\right)  $ is compact, we can take a
finite covering of the form $\left\{  V_{b}\right\}  $ where the $V_{b}$ are
homeomorphic under $\phi_{V}$ to the polydisc $P_{\epsilon}\left(  0\right)  $
in $K^{n}$ defined by $\left\vert y_{i}\right\vert _{K}<\epsilon$, with
$\epsilon$ sufficiently small and for $1\leq i\leq n$. By picking a smooth
partition of the unity subordinate to $\left\{  V_{b}\right\}  $, and using
the previous discussion,
\[
Z_{\Phi}\left(  s,\boldsymbol{f}\right)  =%
%TCIMACRO{\tint \limits_{X_{K}\setminus h^{-1}\left(  \boldsymbol{f}%
%^{-1}\left(  0\right)  \right)  }}%
%BeginExpansion
{\textstyle\int\limits_{X_{K}\setminus h^{-1}\left(  \boldsymbol{f}%
^{-1}\left(  0\right)  \right)  }}
%EndExpansion
\Phi^{\ast}\left(  y\right)  \left\vert \boldsymbol{f}^{\ast}\left(  y\right)
\right\vert _{K}^{s}\left\vert h^{\ast}\left(  dx_{1}\wedge\ldots\wedge
dx_{n}\right)  \right\vert
\]
becomes a finite sum\ of integrals of the following two types:%

\begin{equation}
I\left(  s\right)  :=%
%TCIMACRO{\tint \limits_{K^{n}}}%
%BeginExpansion
{\textstyle\int\limits_{K^{n}}}
%EndExpansion
\Psi\left(  y\right)  \left\vert \boldsymbol{f}^{\ast}\left(  y\right)
\right\vert _{K}^{s}\left\vert dy\right\vert , \label{I}%
\end{equation}
where $\Psi$ is a $C_{0}^{\infty}$ function with support contained in a
polydisc $P_{\epsilon}\left(  0\right)  $ and $e^{s\ln\left\vert
\boldsymbol{f}^{\ast}\left(  y\right)  \right\vert _{K}}$ is $\mathbb{R}%
-$analytic for $y\in V_{b}$ and holomorphic in $s\in\mathbb{C}$, or
\begin{equation}
J\left(  s\right)  :=%
%TCIMACRO{\tint \limits_{K^{n}}}%
%BeginExpansion
{\textstyle\int\limits_{K^{n}}}
%EndExpansion
\Theta\left(  y,s\right)  \left(
%TCIMACRO{\tprod \limits_{i=1}^{r}}%
%BeginExpansion
{\textstyle\prod\limits_{i=1}^{r}}
%EndExpansion
\left\vert y_{i}\right\vert _{K}^{N_{i}s+v_{i}-1}\right)  \left\vert
dy\right\vert , \label{J}%
\end{equation}
where $\Theta\left(  y,s\right)  $ is a $C_{0}^{\infty}$ function with support
contained in a polydisc $P_{\epsilon}\left(  0\right)  $, depending
holomorphically on $s\in\mathbb{C}$. By using the Dominated Convergence Lemma,
we have that (\ref{I}) defines a holomorphic function on the complex plane.
The meromorphic continuation and the description of the corresponding poles
for integrals $J\left(  s\right)  $ is known, see, for instance, the proofs of
Theorem 5.4.1 in \cite{Igusa5} or Theorem 1.6 in \cite{Igusa3}.
\end{proof}

\smallskip We indicate know why also the list of candidate poles in Theorem
\ref{Th1} for $K=\mathbb{R}$ can be derived from the function case. We think
this proposition has some independent interest. A note on terminology: we
still use the term log-principalization starting with one function, i.e. with
a principal ideal; usually one calls this an embedded resolution.

\begin{proposition}
\label{princ-resol} Let $\boldsymbol{f}=(f_{1},\dots,f_{l}):U\rightarrow
\mathbb{R}^{l}$ be an $\mathbb{R}$--analytic mapping on an open $U\subseteq
\mathbb{R}^{n}$. Then $h:X_{\mathbb{R}}\rightarrow U$ is a
log-principalization of the ideal $\mathcal{I}_{\boldsymbol{f}}=(f_{1}%
,\dots,f_{l})$ if and only if it is a log-principalization of the function
$F:=\sum_{i=1}^{l}f_{i}^{2}$. Moreover, when $(N_{i},v_{i}),i\in T,$ are the
numerical data of $h$ for $\mathcal{I}_{\boldsymbol{f}}$, then $(2N_{i}%
,v_{i}),i\in T,$ are the numerical data of $h$ for $F$.
\end{proposition}

\begin{proof}
We suppose that $h$ is a log-principalization of the function $F$, with
numerical data $(2N_{i},v_{i}),i\in T$ ($2N_{i}$ will turn out to be even). We
work in the local ring corresponding to some fixed point of $X_{\mathbb{R}}$;
recall that this local ring is a unique factorization domain. Let $y_{1}%
,\dots,y_{n}$ be coordinates at this point, i.e., a system of parameters of
the local ring. We know that
\[
F^{\ast}(y)=\sum_{i=1}^{l}(f_{i}^{\ast}(y))^{2}=\epsilon(y)\prod_{i=1}%
^{r}y_{i}^{2N_{i}},
\]
where $\epsilon(y)$ is a unit and $r\leq n$. Let $g(y):=\gcd_{1\leq i\leq
l}f_{i}^{\ast}(y)$ and write $f_{i}^{\ast}(y)=g(y)\widetilde{f}_{i}(y)$.

We claim that $\sum_{i=1}^{l}(\widetilde{f}_{i}(y))^{2}$ is a unit $u(y)$.
Assuming the claim, we have that at least one of the $\widetilde{f}_{i}(y)$ is
a unit. Hence
\[
h^{*}\mathcal{I}_{\boldsymbol{f}}=(f_{1}^{*}(y),\dots,f_{l}^{*}(y))=(g(y))
\]
and $\epsilon(y) \prod_{i=1}^{r} y_{i}^{2N_{i}}=u(y)(g(y))^{2}$. So, up to a
unit, $g(y)$ equals $\prod_{i=1}^{r} y_{i}^{N_{i}}$, and indeed $h$ is a
log-principalization of $\mathcal{I}_{\boldsymbol{f}}$.

We now prove the claim. We know that
\[
\epsilon(y)\prod_{i=1}^{r} y_{i}^{2N_{i}} =(g(y))^{2}\left(  \sum_{i=1}%
^{l}(\widetilde{f}_{i}(y))^{2}\right)  .
\]
Since we have unique factorization, either this sum is a unit, or it is
divisible by one of the irreducible elements of the left hand side, i.e., by
$y_{1}$, $y_{2}$, \dots\ or $y_{r}$. Say $\sum_{i=1}^{l}(\widetilde{f}%
_{i}(y))^{2}$ is divisible by $y_{1}$. We can always write this sum as
\[
\sum_{i=1}^{l}(\widetilde{f}_{i}(0,y_{2},\dots,y_{n}))^{2} + y_{1}(\dots).
\]
Then divisibility by $y_{1}$ implies that $\sum_{i=1}^{l}(\widetilde{f}%
_{i}(0,y_{2},\dots,y_{n}))^{2}=0$. Since we are working over $\mathbb{R}$ this
can only happen if $\widetilde{f}_{i}(0,y_{2},\dots,y_{n})=0$ for all
$i=1,\dots,l$. But this is equivalent to $y_{1}$ dividing all $\widetilde
{f}_{i}(y_{1},y_{2},\dots,y_{n})$, contradicting that the $\widetilde{f}%
_{i}(y)$ are relatively prime.

The other implication is quite straightforward.
\end{proof}

\begin{remark}
\textrm{\label{note_log} We recall the description of the (local and global)
log canonical thre\-shold, see e.g. \cite{E-M}-\cite{Kol}-\cite{Sa}, in terms
of a log-principalization. Let $K=\mathbb{R}$ or $\mathbb{C}$ and $U$ an open
submanifold of $K^{n}$. Let $I_{\boldsymbol{f}}=\left(  f_{1},\ldots
,f_{l}\right)  $ be a non-trivial ideal of $K-$analytic functions on $U$. Fix
a log-principalization $h:X_{K}\rightarrow U$ of $I_{\boldsymbol{f}}$ as in
Theorem \ref{theoem1}. }

\textrm{\noindent(1) Let $p\in{\boldsymbol{f}}^{-1}\left(  0\right)  $. The
($K$-)log canonical threshold of $I_{\boldsymbol{f}}$ at $p$ is $c_{p}%
(I_{\boldsymbol{f}})=\min_{i\in T,\text{ }p\in h(E_{i})}\left\{  \frac{v_{i}%
}{N_{i}}\right\}  $. }

\textrm{\noindent(2) The ($K$-)log canonical threshold of $I_{\boldsymbol{f}}$
is $c(I_{\boldsymbol{f}})=\min_{i\in T}\left\{  \frac{v_{i}}{N_{i}}\right\}
$. }
\end{remark}

\begin{proposition}
\label{prop1}Let $\boldsymbol{f}$ and $\Phi$ be as in \ref{fixdata}.

\noindent(1) Let $p\in{\boldsymbol{f}}^{-1}\left(  0\right)  $. If $\Phi$ is
real and nonnegative with support in a small enough neighborhood of $p$ (in
particular $\Phi(p)>0$), then $-c_{p}\left(  \mathcal{I}_{\boldsymbol{f}
}\right)  $ is a pole of $Z_{\Phi}(s,\boldsymbol{f})$, more precisely its
largest pole.

Fix a log-principalization of $\mathcal{I}_{\boldsymbol{f}}$ as in Theorem
\ref{theoem1}.

\noindent(2) Say that $c\left(  \mathcal{I}_{\boldsymbol{f}}\right)  =
\frac{v_{i} }{N_{i}} $ precisely for $i \in T_{\lambda} (\subset T)$. If
$\Phi$ is real and nonnegative and its support intersects $h(\cup_{i\in
T_{\lambda}} E_{i})$, then $-c\left(  \mathcal{I}_{\boldsymbol{f}}\right)  $
is a pole of $Z_{\Phi}(s,\boldsymbol{f})$, more precisely its largest pole.

\noindent(3) Let $r(\mathcal{I}_{\boldsymbol{f}})$ be the maximal number of
$E_{i},i\in T,$ with $c\left(  \mathcal{I}_{\boldsymbol{f}}\right)
=\frac{v_{i}}{N_{i}}$, respectively $c_{p}\left(  \mathcal{I}_{\boldsymbol{f}%
}\right)  =\frac{v_{i}}{N_{i}}$, that have a nonempty intersection. Then
$-c(\mathcal{I}_{\boldsymbol{f}})$, respectively $-c_{p}(\mathcal{I}%
_{\boldsymbol{f}})$, is a pole of order $r(\mathcal{I}_{\boldsymbol{f}})$ of
$Z_{\Phi}(s,\boldsymbol{f})$ when $\Phi$ is as above.
\end{proposition}

\begin{proof}
One uses the case of monomial integrals like in the case of one analytic
function, see e.g. \cite[Chap. II, \S \ 7]{AVG}. The proof is a simple
variation of the one given in \cite[Chap. II, \S \ 7, \ Lemme 4]{AVG} and
\cite[pp. 32-33]{Igusa4} for the case $l=1$. For $K= \mathbb{R}$ one can
alternatively \textit{use} the case $l=1$ and Proposition \ref{princ-resol}.
\end{proof}

Note that this proposition gives an argument to see that these minima do not
depend on the chosen log-principalization.

\begin{corollary}
\label{volume} Let $K$ be $\mathbb{R}$ or $\mathbb{C}$ as before, and
$\boldsymbol{f}$ as in \ref{fixdata}. Let $D$ be a compact subset of $K^{n}$. Then

\noindent(1) $|\boldsymbol{f}(x)|_{K}^{\delta}$ is locally integrable for
$\delta>-c\left(  \mathcal{I}_{\boldsymbol{f}}\right)  $,

\noindent(2) Vol $\left(  \{x\in D\mid|\boldsymbol{f}(x)|_{K}\leq
\alpha\}\right)  \leq\alpha^{c\left(  \mathcal{I}_{\boldsymbol{f}}\right)
-\epsilon}\int_{D}|\boldsymbol{f}(x)|_{K}^{-c\left(  \mathcal{I}%
_{\boldsymbol{f}}\right)  +\epsilon}|dx|$,

\noindent for $\alpha>0$ and any small $\epsilon>0$.
\end{corollary}

\begin{proof}
The first part follows directly from Proposition \ref{prop1}. Then the second
part follows via the Chebyshev inequality.
\end{proof}

Such bounds on volumes have recently emerged as central to aspects of complex
differential geometry, see \cite{P-St} and references therein.

\section{Newton Polyhedra and Log-Principalizations}

We collect some results about Newton polyhedra and log-principalizations
following \cite{V-Z} and the references therein. In this section we take again
$K=\mathbb{R}$ or $K=\mathbb{C}$ .

We set $\mathbb{R}_{+}:=\{x\in\mathbb{R}\mid x\geqslant0\}$. Let $G$ be a
nonempty subset of $\mathbb{N}^{n}$. The \textit{Newton polyhedron }
$\Gamma=\Gamma\left(  G\right)  $ associated to $G$ is the convex hull in
$\mathbb{R}_{+}^{n}$ of the set $\cup_{m\in G}\left(  m+\mathbb{R}_{+}%
^{n}\right)  $. For instance classically one associates a \textit{Newton
polyhedron (at the origin) to } $g(x)=\sum_{m}c_{m}x^{m}$ ($x=\left(
x_{1},\ldots,x_{n}\right)  $, $g(0)=0$), being a nonconstant polynomial
function over $K$ or a $K$--analytic function in a neighborhood of the origin,
where $G=$ supp$(g)$ $:=$ $\left\{  m\in\mathbb{N}^{n}\mid c_{m}\neq0\right\}
$.\ Further we associate more generally a Newton polyhedron to an analytic mapping.

We fix a Newton polyhedron $\Gamma$\ as above. Let $\left\langle \cdot
,\cdot\right\rangle $ denote the usual inner product of $\mathbb{R}^{n}$, and
identify the dual space of $\mathbb{R}^{n}$ with $\mathbb{R}^{n}$ itself by
means of it.

For $a\in\mathbb{R}_{+}^{n}$, we define
\[
d(a,\Gamma)=d(a)=\min_{x\in\Gamma}\left\langle a,x\right\rangle ,
\]
and \textit{the first meet locus }$F(a)$ of $a$ as
\[
F(a):=\{x\in\Gamma\mid\left\langle a,x\right\rangle =d(a)\}.
\]
The first meet locus is a face of $\Gamma$. Moreover, if $a\neq0$, $F(a)$ is a
proper face of $\Gamma$.

We define an equivalence relation in $\mathbb{R}_{+}^{n}$ by taking $a\sim
a^{\prime}\Leftrightarrow F(a)=F(a^{\prime})$. The equivalence classes of
$\sim$ are sets of the form
\[
\Delta_{\tau}=\{a\in\mathbb{R}_{+}^{n}\mid F(a)=\tau\},
\]
where $\tau$ is a face of $\Gamma$.

We recall that the cone strictly\ spanned by the vectors $a_{1},\ldots
,a_{r}\in\mathbb{R}_{+}^{n}\setminus\left\{  0\right\}  $ is the set
$\Delta=\left\{  \lambda_{1}a_{1}+...+\lambda_{r}a_{r}\mid\lambda_{i}
\in\mathbb{R}_{+}\text{, }\lambda_{i}>0\right\}  $. If $a_{1},\ldots,a_{r}$
are linearly independent over $\mathbb{R}$, $\Delta$ is called a
\textit{simplicial cone}. If $\left\{  a_{1},\ldots,a_{r}\right\}  $ is a
subset of a basis of the $\mathbb{Z}$-module $\mathbb{Z}^{n}$, we call
$\Delta$ a \textit{simple cone}.

A precise description of the geometry of the equivalence classes modulo $\sim$
is as follows. Each \textit{facet} (i.e. a face of codimension one) $\gamma$
of $\Gamma$\ has a unique vector $a(\gamma)=(a_{\gamma,1},\ldots,a_{\gamma
,n})\in\mathbb{N}^{n}\mathbb{\setminus}\left\{  0\right\}  $, whose nonzero
coordinates are relatively prime and which is perpendicular to $\gamma$. We
denote by $\mathfrak{D}(\Gamma)$ the set of such vectors. The equivalence
classes are rational cones of the form
\[
\Delta_{\tau}=\{\sum\limits_{i=1}^{r}\lambda_{i}a(\gamma_{i})\mid\lambda
_{i}\in\mathbb{R}_{+}\text{, }\lambda_{i}>0\},
\]
where $\tau$ runs through the set of faces of $\Gamma$, and $\gamma_{i}$,
$i=1,\ldots,r$\ are the facets containing $\tau$. We note that $\Delta_{\tau
}=\{0\}$ if and only if $\tau=\Gamma$. The family $\left\{  \Delta_{\tau
}\right\}  _{\tau}$, with $\tau$ running over\ the proper faces of $\Gamma$,
is a partition of $\mathbb{R}_{+}^{n}\backslash\{0\}$; we call this partition
a \textit{polyhedral subdivision of} $\mathbb{R}_{+}^{n}$ \textit{subordinated
to} $\Gamma$. We call $\left\{  \overline{\Delta}_{\tau}\right\}  _{\tau}$,
the family formed by the topological closures of the $\Delta_{\tau}$, a
\textit{\ fan} \textit{subordinated to} $\Gamma$.

Each cone $\Delta_{\tau}$\ can be partitioned into a finite number of
simplicial cones $\Delta_{\tau,i}$. In addition, the subdivision can be chosen
such that each $\Delta_{\tau,i}$ is spanned by part of $\mathfrak{D}(\Gamma)$.
Thus from the above considerations we have the following partition of
$\mathbb{R}_{+}^{n}\backslash\{0\}$:
\begin{equation}
\mathbb{R}_{+}^{n}\backslash\{0\}=\bigcup\limits_{\tau\text{ }}\left(
\bigcup\limits_{i=1}^{l_{\tau}}\Delta_{\tau,i}\right)  ,\label{simplisubv}%
\end{equation}
where $\tau$ runs over the proper faces of $\Gamma$, and each $\Delta_{\tau
,i}$ is a simplicial cone contained in $\Delta_{\tau}$.

By adding new rays, each simplicial cone can be partitioned further into a
finite number of simple cones. In this way we obtain a \textit{simple
polyhedral subdivision} of $\mathbb{R}_{+}^{n}$ \textit{subordinated to}
$\Gamma$ and a \textit{simple fan} \textit{subordinated to} $\Gamma$ (see e.g.
\cite{K-M-S}).

\subsection{The Newton polyhedron associated to an analytic mapping}

Let $\boldsymbol{f}=(f_{1},\ldots,f_{l})$, $\boldsymbol{f}\left(  0\right)
=0$, be a nonconstant analytic mapping defined on a neighborhood $U\subset
K^{n}$ of the origin. In \cite{V-Z} the authors associated to $\boldsymbol{f}
$ a Newton polyhedron $\Gamma\left(  \boldsymbol{f}\right)  :=\Gamma\left(
\cup_{i=1}^{l}\text{supp}\left(  f_{i}\right)  \right)  $, and a
non-degeneracy condition to $\boldsymbol{f}$ and $\Gamma\left(  \boldsymbol{f}%
\right)  $.

If $f_{i}\left(  x\right)  =
%TCIMACRO{\tsum \nolimits_{m}}%
%BeginExpansion
{\textstyle\sum\nolimits_{m}}
%EndExpansion
c_{m,i}x^{m}$, and $\tau$ is a face of $\Gamma\left(  \boldsymbol{f}\right)
$, we set%
\[
f_{i,\tau}\left(  x\right)  :=
%TCIMACRO{\tsum \limits_{m\in\text{supp}(f_{i})\cap\tau}}%
%BeginExpansion
{\textstyle\sum\limits_{m\in\text{supp}(f_{i})\cap\tau}}
%EndExpansion
c_{m,i}x^{m}.
\]

\begin{definition}
\textrm{\label{def1} Let $\boldsymbol{f}=(f_{1},\ldots,f_{l}):U\longrightarrow
K^{l}$ be a nonconstant analytic mapping satisfying $\boldsymbol{f}\left(
0\right)  =0$. The mapping \ $\boldsymbol{f}$ is called \textit{strongly
non-degenerate at the origin with respect to }$\Gamma(\boldsymbol{f}
)$,\textit{\ }if for any \textit{compact} face $\tau\subset\Gamma
(\boldsymbol{f})$ and any $z\in\left\{  z\in\left(  K^{\times}\right)
^{n}\mid f_{1,\tau}(z)=\ldots=f_{l,\tau}(z)=0\right\}  $ it satisfies
$rank\left[  \frac{\partial f_{i,\tau}}{\partial x_{j}}\left(  z\right)
\right]  =\min\{l,n\}$.
%(2) Let $\boldsymbol{f}=(f_{1},\ldots,f_{l}):K^{n}\longrightarrow
%K^{l}$ be a nonconstant polynomial mapping satisfying $\boldsymbol{f}\left(
%0\right)  =0$. The mapping $\boldsymbol{f}$\ is called \textit{strongly non-degenerate with
%respect to }$\Gamma(\boldsymbol{f})$,\textit{\ }if for any face $\tau\subset\Gamma(\boldsymbol{f})$, including
%$\Gamma(\boldsymbol{f})$ itself, and  any $z\in\left\{  z\in\left(
%K^{\times}\right)  ^{n}\mid f_{1,\tau}(z)=\ldots=f_{l,\tau}(z)=0\right\}$
%it satisfies $rank_{K}\left[  \frac{\partial f_{i,_{\tau}}}{\partial
%x_{j}}\left(  z\right)  \right]  $ $=$ \ $\min\{l,n\}$.
}
\end{definition}

\begin{remark}
\textrm{(1) The above notion of non-degeneracy agrees with the one given by
Varchenko for the case $l=1$, see \cite{Var}. On the other hand, the previous
notion does not agree with the non-degeneracy notion with respect to a
collection of Newton polyhedra given by Khovanskii in \cite{K}. We refer the
reader to \cite{V-Z}\ for a further discussion about the mentioned
non-degeneracy conditions. }

\textrm{\noindent(2) If we fix $l$ and $\Gamma$, then one can show, just as
for the classical case $l=1$, that `most' mappings $\boldsymbol{f}%
=(f_{1},\ldots,f_{l})$, with $\boldsymbol{f}(0)=0 $ and Newton polyhedron
$\Gamma(\boldsymbol{f})=\Gamma$, are strongly non-degenerate at the origin.
One can easily generalize e.g. the proof of Lemme 1 in \cite[p. 157]{AVG}.
However, just as for $l=1$, we should say that there are many interesting
non-generic mappings. }

\textrm{\noindent(3) In \cite{K}, Khovanskii established the existence of an
embedded resolution for a variety using a collection of Newton polyhedra, see
also \cite{AVG} for the case $l=1$. In \cite{V-Z}, a log-principalization for
an ideal with generators satisfying the above-mentioned notion of
non-degeneracy and using one Newton polyhedron was established, see
Proposition \ref{proposition1} below. This result agrees with Khovanskii's
result only in the case $l=1$. }
\end{remark}

\begin{remark}
\textrm{\label{note} When $K=\mathbb{R}$ it is again natural to try to study
the mapping $\boldsymbol{f}=(f_{1},\ldots,f_{l})$ via the function
$F:=\sum_{i=1}^{l}f_{i}^{2}$. It is not difficult to verify that $\Gamma(F)$
is the `double'\ of $\Gamma(\boldsymbol{f})$, i.e. obtained from it after
scaling by a factor $2$. Consider however the statements }

\textrm{(i) $F$ is non-degenerate at the origin with respect to $\Gamma(F)$,
and }

\textrm{(ii) $\boldsymbol{f}$ is strongly non-degenerate at the origin with
respect to $\Gamma(\boldsymbol{f})$. }

\textrm{\noindent It is easy to verify that (i) implies (ii), but in general
the converse is not true. Consider for instance any strongly non-degenerate
$\boldsymbol{f}$ for which $\{f_{1,\tau}=\ldots=f_{l,\tau}=0\}\cap\left(
{\mathbb{R}}^{\times}\right)  ^{n}\neq\emptyset$ for some compact face $\tau$
of $\Gamma(\boldsymbol{f})$, e.g. $\boldsymbol{f}=(x^{2}-y^{3},x^{2}-z^{3})$.
}

\textrm{In such cases the `classical'\ embedded resolution of a non-degenerate
function is not helpful, but one can use the log-principalization of a
strongly non-degenerate mapping of Proposition \ref{proposition1} below. }
\end{remark}

\subsection{Log-principalizations}

\begin{proposition}
[{\cite[Prop. 3.9]{V-Z}}]\label{proposition1}Let $\boldsymbol{f}=(f_{1}%
,\ldots,f_{l}):U(\subset K^{n})\longrightarrow K^{l}$ be a nonconstant
analytic mapping, strongly \textit{non-degenerate at the origin with respect
to }$\Gamma\left(  \boldsymbol{f}\right)  $. Let $\ \mathcal{F}
_{\boldsymbol{f}}$ be a simple fan subordinated to $\Gamma\left(
\boldsymbol{f}\right)  $. Let $Y_{K}$ be the toric manifold corresponding to
$\mathcal{F}_{\boldsymbol{f}}$, and let
\[
\sigma_{0}:Y_{K}\longrightarrow U
\]
be the restriction of the corresponding toric map to the inverse image of $U$.
Denote by $Z$ the set of common zeroes of $\mathcal{I}_{\boldsymbol{f}}
=(f_{1},\ldots,f_{l})$ in $U\cap\left(  K^{\times}\right)  ^{n}$. When $U$ is
taken small enough, either $Z=\varnothing$ or it is a submanifold of
codimension $l$. In this last case we have $l<n$ and we denote the closure of
$Z$ in $Y_{K}$ by $Z_{Y}$.

\noindent(1) If $Z=\varnothing$ (or if $l=1$), the ideal $\sigma_{0}^{\ast
}\left(  \mathcal{I}_{\boldsymbol{f}}\right)  $ is principal (and monomial) in
a sufficiently small neighborhood of $\sigma_{0}^{-1}\left\{  0\right\}  $.

\noindent(2) If $Z\neq\varnothing$, we have that $Z_{Y}$ is a closed
submanifold of $Y_{K}$, having normal crossings with the exceptional divisor
of $\sigma_{0}$. Let $\sigma_{1}:X_{K}\longrightarrow Y_{K}$ be the blowing-up
of $Y_{K}$\ with center $Z_{Y}$, and let $\sigma=\sigma_{0}\circ\sigma
_{1}:X_{K}\longrightarrow U$.\ Then the ideal $\sigma^{\ast}\left(
\mathcal{I}_{\boldsymbol{f}}\right)  $ is principal\ (and monomial) in a
sufficiently small neighborhood of $\sigma^{-1}\left\{  0\right\}  $.
\end{proposition}

\section{Poles for Local Zeta Functions and Newton Polyhedra}

Given $\xi=\left(  \xi_{1},\ldots,\xi_{n}\right)  \in\mathbb{N}^{n}
\setminus\left\{  0\right\}  $, we put $\sigma\left(  \xi\right)  :=\xi
_{1}+\ldots+\xi_{n}$ and $d\left(  \xi\right)  =\min_{x\in\Gamma\left(
\boldsymbol{f}\right)  }\left\langle \xi,x\right\rangle $ as before. We say
that $\xi$ is a primitive vector, if $\gcd\left(  \xi_{1},\ldots,\xi
_{n}\right)  =1$. If $d\left(  \xi\right)  \neq0$, we define
\[
\mathcal{P}\left(  \xi\right)  =\left\{  -\frac{\sigma\left(  \xi\right)
+k}{d\left(  \xi\right)  } \mid k\in\mathbb{N} \right\}  .
\]
We also define
\[
\gamma_{0}\left(  \boldsymbol{f}\right)  =\min_{\xi\in\mathfrak{D}%
(\Gamma\left(  \boldsymbol{f}\right)  )}\left\{  \frac{\sigma\left(
\xi\right)  }{d\left(  \xi\right)  }\right\}  .
\]
Varchenko called $\gamma_{0}\left(  \boldsymbol{f}\right)  $ the distance from
the origin to $\Gamma\left(  \boldsymbol{f}\right)  $. The number $\gamma
_{0}\left(  \boldsymbol{f}\right)  $ admits the following geometric
interpretation. Let $\left(  t_{0},\ldots,t_{0}\right)  $ be the intersection
point of the diagonal $\{\left(  t,\ldots,t\right)  \in\mathbb{R}^{n}\mid
t\in\mathbb{R}\}$ with the boundary of $\Gamma\left(  \boldsymbol{f}\right)
$, then $\gamma_{0}\left(  \boldsymbol{f}\right)  =1/t_{0}$.
%Let
%$\tau_{\boldsymbol{f}}$ be the smallest face of $\Gamma(\boldsymbol{f}
%)$\ containing $\left(  t_{0},\ldots,t_{0}\right)  .$ We set $\kappa
%_{\boldsymbol{f}}$ for the codimension of $\tau_{\boldsymbol{f}}$ in
%$\mathbb{R}^{n}$.

Let $\ \mathcal{F}_{\boldsymbol{f}}$ be a simple fan subordinated to
$\Gamma\left(  \boldsymbol{f}\right)  $. Then the set of generators of the
cones in $\mathcal{F}_{\boldsymbol{f}}$, i.e. the skeleton of $\mathcal{F}
_{\boldsymbol{f}}$, can be partitioned as $\Lambda_{\boldsymbol{f}}
\cup\mathfrak{D}(\Gamma\left(  \boldsymbol{f}\right)  )$, where $\Lambda
_{\boldsymbol{f}}$ is a finite set of primitive vectors, corresponding to the
extra rays, induced by the subdivision into simple cones.

The numerical data of the log-principalizations constructed in Proposition
\ref{proposition1} can be computed directly from the explicit expressions for
the generators of $\sigma_{0}^{\ast}\left(  \mathcal{I}_{\boldsymbol{f}%
}\right)  $, $\sigma^{\ast}\left( \mathcal{I}_{\boldsymbol{f}}\right)  $, and
Lemme 8 in \cite[p. 201]{AVG}. The following theorem follows from the previous
considerations by adapting the proof given by Varchenko for the case $l=1$ to
arbitrary $l\geq1$.

\begin{theorem}
\label{Theorem2a} Let $\boldsymbol{f}=(f_{1},\ldots,f_{l}):U(\subset K^{n}
)\longrightarrow K^{l}$, with $\boldsymbol{f}\left(  0\right)  =0$, be an
analytic mapping, strongly non-degenerate at the origin with respect to
$\Gamma\left(  \boldsymbol{f}\right) $. There exists a neighborhood $V
(\subset U)$ of the origin such that, if $\Phi$ is a smooth function with
support contained in $V$, then the following assertions hold.

\noindent(1) The function $Z_{\Phi}\left(  s,\boldsymbol{f}\right)  $ is
holomorphic on the complex half-plane $\operatorname{Re}(s)>$\linebreak%
$\max\left\{  -\gamma_{0}\left(  \boldsymbol{f}\right) ,-l\right\} $.

\noindent(2) The poles of $Z_{\Phi}\left(  s,\boldsymbol{f} \right)  $ belong
to the set $\cup_{\xi\in\mathfrak{D} (\Gamma\left(  \boldsymbol{f}\right)
)}\mathcal{P}\left(  \xi\right)  \cup\cup_{\xi\in\Lambda_{\boldsymbol{f}}%
}\mathcal{P}\left(  \xi\right)  \cup\left(  -\left(  l+\mathbb{N}\right)
\right)  $, where the last set may be discarded if $l\geq n$.

\noindent(3) If $\gamma_{0}\left(  \boldsymbol{f}\right)  <l$, then
$s=-\gamma_{0}\left(  \boldsymbol{f}\right)  $ is a pole of $Z_{\Phi}\left(
s,\boldsymbol{f}\right) $ as a distribution on the space of smooth functions
with compact support.
\end{theorem}

In the case $l=1$ and $K=\mathbb{R}$, Denef and Sargos proved in \cite{D-S}
that $\cup_{\xi\in\Lambda_{\boldsymbol{f}}}\mathcal{P}\left(  \xi\right)  $
may be discarded from the list of candidate poles in Theorem \ref{Theorem2a}.
This is a strong and interesting result, yielding in general a much shorter
list of candidate poles, that can moreover be read off immediately from
$\Gamma\left(  \boldsymbol{f}\right) $.

In the next sections, we extend the Denef-Sargos result to arbitrary $l \geq
1$. Actually, we follow reasonably closely the approach of \cite{D-S}.
However, a number of extra difficulties pop up, for which a careful analysis
is needed. In order to produce a readable text, we have to recall in the
sequel the main ideas of \cite{D-S}. Our goal will be to show the following
result, for which we already indicate the starting point of its proof.

\begin{theorem}
\label{Theorem3} Let $\boldsymbol{f}=(f_{1},\ldots,f_{l}):U (\subset
\mathbb{R}^{n}) \longrightarrow\mathbb{R}^{l}$, with $\boldsymbol{f}\left(
0\right)  =0$, be an analytic mapping, strongly non-degenerate at the origin
with respect to $\Gamma\left(  \boldsymbol{f}\right) $. There exists a
neighborhood $V (\subset U)$ of the origin such that, if $\Phi$ is a smooth
function with support contained in $V$, then the following assertions hold.

\noindent(1) The poles of $Z_{\Phi}\left(  s,\boldsymbol{f}\right)  $ belong
to the set $\cup_{\xi\in\mathfrak{D}(\Gamma\left(  \boldsymbol{f}\right)
)}\mathcal{P}\left(  \xi\right)  \cup\left(  -\left(  l+\mathbb{N}\right)
\right)  $.

\noindent(2) Let $\rho$ be an integer satisfying $1\leq\rho\leq n$, and let
$s_{0}$ be a candidate pole of $Z_{\Phi}\left(  s,\boldsymbol{f}\right)  $
with $s_{0}\notin-\left(  l+\mathbb{N}\right)  $ (respectively $s_{0}%
\in-\left(  l+\mathbb{N}\right)  $). A necessary condition for $s_{0}$ to be a
pole of $Z_{\Phi}\left(  s,\boldsymbol{f}\right) $ of order $\rho$, is that
there exists a face $\tau\subset\Gamma\left(  \boldsymbol{f}\right)  $ of
codimension $\rho$ (respectively of codimension $\rho-1$) such that $s_{0}%
\in\mathcal{P}\left(  \xi\right)  $ for any facet $F\left(  \xi\right)
\supset\tau$.

%In the case in which $\boldsymbol{f}$ is strongly non-degenerate at the origin
%with respect to $\Gamma\left(  \boldsymbol{f}\right)  $, then all the above
%assertions hold if $U$ is a sufficiently small neighborhood of the origin.

\end{theorem}

\begin{proof}
Define
\[
I_{\Phi}(s,\boldsymbol{f})=I(s,\boldsymbol{f},\Phi)=\int_{\mathbb{R}_{+}^{n}%
}\Phi(x)\left\vert \boldsymbol{f}(x)\right\vert _{\mathbb{R}}^{s}\,\left\vert
dx\right\vert \text{, }\operatorname{Re}(s)>0\text{.}%
\]
This integral defines an analytic function for $\operatorname{Re}(s)>0$;
$I(s,\boldsymbol{f},\Phi)$ and $Z_{\Phi}(s,\boldsymbol{f})$ are related by
\[
Z_{\Phi}(s,\boldsymbol{f})=\sum_{\theta\in\{-1,1\}^{n}}I(s,\boldsymbol{f}%
(\theta\cdot x),\Phi(\theta\cdot x)),
\]
where $\theta\cdot x=(\theta_{1}x_{1},\ldots,\theta_{n}x_{n})$. The result
will follow from the meromorphic continuation of $I_{\Phi}(s,\boldsymbol{f})$,
and the explicit description of its poles, cf. Theorem \ref{Theorem5}.
\end{proof}

\section{Monomial integrals and Decoupages}

From now on we take $K=\mathbb{R}$ and use all the notations introduced in the
previous sections.

\subsection{Some Monomial Integrals}

We give some results about the meromorphic continuation of integrals attached
to monomials that we will use later on. These results are easy variations of
well-known results, see e.g. \cite[Chap. II, \S \ 7, \ Lemme 3]{AVG},
\cite[Lemme 3.1]{D-S}, \cite[Chap. I, Sect. 3.2]{G-S}, and \cite[pp.
101-102]{Igusa3}.

\begin{lemma}
\label{Lema1} Let $\Omega$ be an open neighborhood of the origin in
$\mathbb{R}^{k}\times\mathbb{R}$ and let $g:\Omega\rightarrow\mathbb{R}$ be an
$\mathbb{R}$-analytic function. Let $\phi$ be a smooth function with support
in $\Omega$ and containing the origin. Take $a=\left(  a_{1},\ldots
,a_{k+1}\right)  \in\left(  \mathbb{R}_{+}\mathbb{\smallsetminus}\left\{
0\right\}  \right)  ^{k+1}$ such that $[0,a_{1}]\times\cdots\times
\lbrack0,a_{k+1}]$ is contained in the support of $\phi$. Assume that
$g(y,z)>0$ for $\left(  y,z\right)  $ in the support of $\phi$, and define
\[
J(s)=\int_{0}^{a_{1}}\dots\int_{0}^{a_{k+1}}\left(  \prod_{j=1}^{r}
y_{j}^{sm_{j}+\gamma_{j}-1}\right)  z^{s+l-1}g(y,z)^{s}\phi(y,z)\,\left\vert
dy\wedge dz\right\vert ,\text{ }\operatorname{Re}(s)>0,
\]
with $1\leq r\leq k$ and $m_{j},\gamma_{j}\in\mathbb{N}\setminus\{0\}$ for
$j=1,\ldots,r$, and $l\in\mathbb{N}\setminus\{0\}$. Then the following
assertions hold:

\noindent(1) $J(s)$ is convergent and defines a holomorphic function on
\[
\operatorname{Re}(s)>\max\{-l,-\gamma_{1}/m_{1},\dots,-\gamma_{r}/m_{r}\};
\]

\noindent(2) $J(s)$ admits a meromorphic continuation to the whole complex
plane, with poles of order at most $k+1$. Furthermore, the poles belong to
\[
\bigcup_{1\leq i\leq r}\left(  -\frac{\gamma_{i}+\mathbb{N}}{m_{i}}\right)
\cup\left(  -\left(  l+\mathbb{N}\right)  \right)  .
\]

\noindent(3) Let $\rho$ be a positive integer and let $s_{0}$ be a candidate
pole of $J(s)$ with $s_{0}\notin-\left(  l+\mathbb{N}\right) $ (resp.
$s_{0}\in-\left(  l+\mathbb{N}\right) $). A necessary condition for $s_{0}$ to
be a pole of $J(s)$ of order $\rho$, is that
\[
Card\left\{  i\mid s_{0}\in\frac{-\left(  \gamma_{i}+\mathbb{N}\right)
}{m_{i}}\right\}  \geq\rho\ \text{(resp.}\ \geq\rho-1\text{)} .
\]

\end{lemma}

\medskip

\begin{corollary}
\label{Cor1} Let $r$, $k$ and $l$ be natural numbers such that $1\leq r\leq k$
and $k+1-r-l\geq0$. Let $U$ be a neighborhood of the origin of $\mathbb{R}%
^{l}$. Let $\theta$ be a smooth function with compact support contained in
$[0,1]^{r}\times U\times\lbrack0,1]^{k+1-r-l}$. Set
\[
J_{1}(s):=\int\limits_{[0,1]^{r}\times U\times\lbrack0,1]^{k+1-r-l}}%
\prod_{j=1}^{r}y_{j}^{sm_{j}+\gamma_{j}-1}\left(  \sum_{i=r+1}^{r+l}y_{i}%
^{2}\right)  ^{s/2}\theta(y)\,\left\vert dy\right\vert \text{, }%
\operatorname{Re}(s)>0.
\]
Then $J_{1}(s)$ has a meromorphic continuation to $\mathbb{C}$ and its poles
satisfy all the conclusions in Lemma \ref{Lema1}.
\end{corollary}

\begin{proof}
The result follows from the previous lemma by using hyper-spherical coordinates.
\end{proof}

\subsection{Decoupages in $\mathbb{R}^{n}$ and compatible mappings}

\label{Section comp}

Let $\boldsymbol{f}=(f_{1},\ldots,f_{l}):U(\subset\mathbb{R}^{n}%
)\longrightarrow\mathbb{R}^{l}$ be an analytic mapping, strongly
non-degenerate at the origin with respect to $\Gamma\left(  \boldsymbol{f}%
\right)  $.

%Our goal is to give a short list for the candidate poles for
%$Z_{\Phi}(s,\boldsymbol{f})$, when $l\leq n$ and the support of $\Phi$ is
%contained in $\left[  -1,1\right]  ^{n}$, by using the technique of decoupages
%presented by Denef and Sargos in \cite{D-S}. In this section we review the
%notion of \textit{decoupage} and the construction of decoupages of $\left[
%0,1\right]  ^{n}$ from a simplicial fan subordinated to a Newton polyhedron,
%following \cite{D-S}.

We fix a simplicial fan $\mathcal{F}_{\boldsymbol{f}}$\ subordinated to
$\Gamma(\boldsymbol{f})$, and assume that the cones of $\mathcal{F}%
_{\boldsymbol{f}}$ of dimension $n$ are $\overline{\Delta}_{1},\ldots
,\overline{\Delta}_{m}$.

\begin{definition}
\textrm{Let $X\subseteq\mathbb{R}^{n}$ be a measurable set in the sense of
Lebesgue. A \textit{decoupage} of $X$ is a finite family $\mathcal{D}%
=\{D_{1},\ldots,D_{m}\}$ of measurable sets satisfying }

\textrm{\noindent(1) $D_{i}\cap D_{j}$ has measure zero for $i\neq j$, and }

\textrm{\noindent(2) $X\smallsetminus(\cup_{i=1}^{m}D_{i})$ and $(\cup
_{i=1}^{m}D_{i})\smallsetminus X$ have measure zero. }
\end{definition}

\smallskip Define
\[
\mathcal{L}:(0,1]^{n}\rightarrow\mathbb{R}_{+}^{n}:(x_{1},\ldots
,x_{n})\rightarrow(-\ln x_{1},\ldots,-\ln x_{n}).
\]
Then $\mathcal{L}$ is an $\mathbb{R}$-analytic isomorphism, and $\{\mathcal{L}%
^{-1}(\overline{\Delta}_{j})\}_{1\leq j\leq m}$ is a decoupage of $[0,1]^{n}$.
In addition, for any smooth function $\phi(x)$,
\[
\int_{\lbrack0,1]^{n}}\phi(x)\,\left\vert dx\right\vert =\sum_{j=1}^{m}%
\int_{\mathcal{L}^{-1}(\overline{\Delta}_{j})}\phi(x)\,\left\vert
dx\right\vert .
\]

%\subsection{Compatible mappings with a given decoupage}
%\label{Section comp}
\smallskip In the sequel we will use a decoupage of the domain of integration
to give a short list of candidate poles for local zeta functions. It may
happen that a mapping $\left(  f_{1,\tau}\left(  x\right)  ,\ldots,f_{l,\tau
}\left(  x\right)  \right)  $, restricted to certain components of the
boundary of a decoupage, has singularities. These cases will turn out be to
annoying, but they are rare and they can be removed by using homotheties, cf.
Lemma \ref{compatibility}.

Take vectors $a_{1},\ldots,a_{n}$ $\in\mathbb{N}^{n}$ that are linearly
independent over $\mathbb{R}$, and denote
\[
\overline{\Delta}:=\{ \sum_{i=1}^{n}\lambda_{i}a_{i}\mid\lambda_{i}
\in\mathbb{R}_{+}\}
\]
for the closed cone spanned by them.
%Set $S:=\mathcal{L}^{-1}(\overline {\Delta})$
For $K\subset\{1,\ldots,n\}$, set
\[
\Delta_{K}:=\{\sum_{i\in K}\lambda_{i}a_{i}\mid\lambda_{i}\in\mathbb{R}
_{+},\ \lambda_{i}>0\}
\]
and $S_{K}:=\mathcal{L}^{-1}(\Delta_{K})$, with the convention that
$S_{K}=\{(1,\ldots,1)\}$ if $K=\emptyset$. We will call $S_{K}$ \textit{a
sector}.

\begin{definition}
\textrm{\label{def2} We say that $\boldsymbol{f}$ \textit{is compatible with
the sector} $S_{K}$ if the following condition is satisfied. For each
partition $\{I,J,K\}$ of $\{1,\ldots,n\}$, with $I\neq\emptyset,\ J\neq
\emptyset$, and such that the face $\tau=\cap_{i\in I}F(a_{i})$ is compact and
non-empty, the restriction to $S_{K}$ of the map $\boldsymbol{f}_{\tau
}:=\left(  f_{1,\tau},\dots,f_{l,\tau}\right) $ does not admit $\left(
0,\ldots,0\right)  \in\mathbb{R}^{l}$ as a critical value. }

\textrm{We say that $\boldsymbol{f}$ is compatible with $\mathcal{L}%
^{-1}(\overline{\Delta})$, if $\boldsymbol{f}$ is compatible with the sector
$S_{K}$ for any $K\subset\{1,\ldots,n\}$. }

\textrm{We say that $\boldsymbol{f}$ is \textit{compatible with the decoupage}
$\mathcal{L}^{-1}(\mathcal{F}_{\boldsymbol{f}})=\{\mathcal{L}^{-1}%
(\overline{\Delta})\mid\overline{\Delta}\in\mathcal{F}_{\boldsymbol{f}}$ with
$\dim\overline{\Delta}=n\}$ (or \textit{compatible with} $\mathcal{F}%
_{\boldsymbol{f}}$\textit{\ }), if $\boldsymbol{f}$ is compatible with
$\mathcal{L}^{-1}(\overline{\Delta})$, for any $n$-dimensional cone
$\overline{\Delta}$ in $\mathcal{F}_{\boldsymbol{f}}$. }
\end{definition}

\begin{remark}
\textrm{Note that the condition `the restriction to $S_{K}$ of $\boldsymbol{f}%
_{\tau}$ does not admit $\left(  0,\ldots,0\right)  \in\mathbb{R}^{l}$ as a
critical value' is equivalent to the two following conditions: }

\textrm{\noindent(1) if $Card(K)\geq l$, then $rank \left[  \frac{\partial
f_{i,\tau}}{\partial x_{j}}(z)\right]  =l,$ for any $z\in\{z\in S_{K}\mid
f_{1,\tau}(z)=\dots=f_{l,\tau}(z)=0\}$, and }

\textrm{\noindent(2) if $0\leq Card(K)\leq l-1$, then for any $z_{0}\in S_{K}%
$, there exists an index $i_{0}=i_{0}(z_{0})$, such that $f_{i_{0},\tau}%
(z_{0})\neq0.$ }
\end{remark}

\begin{remark}
\textrm{\label{note1}Set $n=l=2$. Since any partition $\{I,J,K\}$ of
$\{1,2\}$, with $I\neq\emptyset$ and $\ J\neq\emptyset$, requires that
$K=\emptyset$ and that $Card(I)=1$, any $\tau$ in Definition \ref{def2} is a
facet of $\Gamma(\boldsymbol{f})$. Therefore $\boldsymbol{f}$ is compatible
with a simplicial fan $\mathcal{F}_{\boldsymbol{f}}$ if and only if
$f_{1,\gamma}(1,1)\neq0$ or $f_{2,\gamma}(1,1)\neq0$, for any compact facet
$\gamma$ of $\Gamma(\boldsymbol{f})$. }
\end{remark}

\begin{example}
\textrm{Set $\boldsymbol{f}=(x_{2}^{a}-x_{1}^{b},x_{1}^{c}-x_{2}^{d})$, with
$a\leq d$ and $b\leq c$. Then $\boldsymbol{f}$ is not compatible with any
simplicial fan $\mathcal{F}_{\boldsymbol{f}}$, since $f_{i,\gamma}(1,1)=0$
($i=1,2$) for the facet $\gamma$ containing $(0,a)$ and $(b,0)$. Set
$T_{(3,2)}:(\mathbb{R}_{+}\setminus\{0\})^{2}\rightarrow(\mathbb{R}%
_{+}\setminus\{0\})^{2} ;\ (x_{1},x_{2})\mapsto(3x_{1},2x_{2})$. Then
$\boldsymbol{f}\circ T_{(3,2)}$ is compatible with every simplicial fan
$\mathcal{F}_{\boldsymbol{f}}$. }
\end{example}

\begin{example}
\textrm{Set $\mathbf{g}=(x_{2}^{a}+x_{1}^{b},x_{1}^{c})$, with $c\leq b$. Then
$\mathbf{g}$ is compatible with every simplicial fan $\mathcal{F}%
_{\boldsymbol{f}}$. }
\end{example}

\subsection{Compatibility and homotheties}

Subsequently we will use only the part of the following lemma, stating that
the sets $T_{l,k^{\prime}}$ have measure zero. But we think it is natural to
state and prove it as below.

\begin{lemma}
\label{sard} Assume $l\leq n$. Let $W_{0}\times W_{1}$ be an open neighborhood
of $(0,0)\in\mathbb{R}^{n}\times\mathbb{R}^{n}$, and $\mathbf{h}:W_{0}\times
W_{1}\rightarrow\mathbb{R}^{l} : (t,z)\mapsto\mathbf{h}(t,z)=(h_{1}
(t,z),\ldots,h_{l}(t,z))$ an ana\-ly\-tic mapping such that $(0,\ldots
,0)\in\mathbb{R}^{l}$ is not a critical value of $\mathbf{h}$. Set
\[
M:=\left\{  (t,z)\in W_{0}\times W_{1}\mid h_{1}(t,z)=\dots=h_{l}%
(t,z)=0\right\}  ,
\]
$\pi:M\rightarrow\mathbb{R}^{n} : (t,z)\mapsto t$, and
\[
T_{l^{\prime},k^{\prime}}:=\left\{
\begin{array}
[c]{l}%
t\in W_{0}\mid\mathbf{h}(t,z)=0\text{ for some }z\in W_{1}\text{ and }\\
rank \left[  \frac{\partial h_{i}}{\partial z_{k}}(t,z)\right]
_{\substack{1\leq i\leq l^{\prime}\\k^{\prime}\leq k\leq n}}<\min\{l^{\prime
},n-k^{\prime}+1\}
\end{array}
\right\}
\]
for $1\leq l^{\prime}\leq l$ and $1\leq k^{\prime}\leq n$. Then $T_{l^{\prime
},k^{\prime}}$ has measure zero, and the set of critical values of $\pi$ is
the union of the sets $T_{l^{\prime},k^{\prime}}$ and of the similar sets
obtained by permuting the indices.
\end{lemma}

\begin{proof}
We establish the result by showing explicitly that the set of critical values
of $\pi$ restricted to a neighborhood $M_{(a,b)}$ of point $\left(
a,b\right)  $ in $M$ has the form $T_{l^{\prime},k^{\prime}}$. Then by Sard's
lemma $T_{l^{\prime},k^{\prime}}$ has measure zero, for any $l^{\prime
},k^{\prime}$. This explicit description is achieved as follows: given
$(a,b)\in M$, by the implicit function theorem, we can solve the system
$\left\{  h_{1}(t,z)=\dots=h_{l}(t,z)=0\right\}  $ locally around $\left(
a,b\right)  $, for $l$ variables in terms of the others. Depending on the
selection made upon the set of variables several cases occur.

\textbf{Case 1.} All the depending variables are taken from the first $n$
variables. After renaming the variables if necessary, we assume that%

\[
\det\left[  \frac{\partial h_{i}}{\partial t_{j}}\left(  (a,b)\right)
\right]  _{\substack{1\leq i\leq l\\1\leq j\leq l}}\neq0.
\]

Set $t=(t^{\prime},t^{\prime\prime})$, where $t^{\prime}=(t_{1},\ldots,t_{l})$
and $t^{\prime\prime}=(t_{l+1},\ldots,t_{n})$. We also set $(a,b)=(a^{\prime
},a^{\prime\prime},b)\in(\mathbb{R}^{l}\times\mathbb{R}^{n-l}\times
\mathbb{R}^{n})\cap M$. Then there exist an open set $G\subset\mathbb{R}%
^{2n-l}$, containing $(a^{\prime\prime},b)$, and a unique analytic function
$\mathbf{g}:G\rightarrow\mathbb{R}^{l}$, such that $\mathbf{g}(a^{\prime
\prime},b)=a^{\prime}$ and $\mathbf{h}(\mathbf{g} (t^{\prime\prime
},z),t^{\prime\prime},z)=0$ for any $(t^{\prime\prime},z)\in G$.

Set $M_{(a,b)}=\{(t^{\prime},t^{\prime\prime},z)\in\mathbf{g}(G)\times G\mid
t_{1}=g_{1}(t^{\prime\prime},z),\ldots,t_{l}=g_{l}(t^{\prime\prime
},z)\}\subseteq M$. Then%

\[%
\begin{array}
[c]{cccc}%
\pi\mid_{M_{(a,b)}}: & M_{(a,b)} & \rightarrow & \mathbb{R}^{n}\\
&  &  & \\
& \left( t^{\prime\prime},z\right)  & \mapsto & \left(  g_{1}(t^{\prime\prime
},z),\ldots,g_{l}(t^{\prime\prime},z),t^{\prime\prime}\right)  ,
\end{array}
\]
and the Jacobian matrix of $\pi\mid_{M_{(a,b)}}$ is
\[
J(t^{\prime\prime},z)=\left[
\begin{matrix} \frac{\partial g_1}{\partial t_{l+1}}&\dots&\frac{\partial g_1}{\partial t_{n}}&\frac{\partial g_1}{\partial z_{1}}&\dots&\frac{\partial g_1}{\partial z_{n}}\\ \hdotsfor{6}\\ \frac{\partial g_l}{\partial t_{l+1}}&\dots&\frac{\partial g_l}{\partial t_{n}}&\frac{\partial g_l}{\partial z_{1}}&\dots&\frac{\partial g_l}{\partial z_{n}}\\
1&\dots&0&0&\dots&0\\ \hdotsfor{6}\\ 0&\dots&1&0&\dots&0 \end{matrix}\right]
(t^{\prime\prime},z).
\]
The point $(t^{\prime},t^{\prime\prime},z)$ is a critical point for $\pi$ if
$rank J(t^{\prime\prime},z)<n$; this means that $rank \left[  \frac{\partial
g_{i}}{\partial z_{k}}(t^{\prime\prime},z)\right]  _{\substack{1\leq i\leq
l\\1\leq k\leq n}}<l$. Since $\sum_{j=1}^{l}\left(  \frac{\partial h_{i}%
}{\partial t_{j}}\right)  \left(  \frac{\partial g_{j}}{\partial z_{k}%
}\right)  = \frac{\partial h_{i}}{\partial z_{k}}$ and the matrix
\linebreak$\left[  \frac{\partial h_{i}}{\partial t_{j}}\right]
_{\substack{1\leq i\leq l\\1\leq j\leq l}}$ has maximal rank around $(a,b)$,
this last condition is equivalent to
\[
%rank\left[  \frac{\partial g_{i}}{\partial z_{k}}(a^{\prime\prime},b)\right]  _{\substack{1\leq i\leq l\\1\leq k\leq
%n}}=
rank\left[  \frac{\partial h_{i}}{\partial z_{k}}((t^{\prime},t^{\prime\prime
},z))\right]  _{\substack{1\leq i\leq l\\1\leq k\leq n}}<l.
\]
Hence $T_{l,1}$ is the set of critical values of $\pi\mid_{M_{(a,b)}}$.

\textbf{Case 2}. Some of the depending variables are taken from the first $n$
variables ($t$-variables) and the rest are taken from the last $n$ variables
($z$-variables). After renaming the variables if necessary, we assume that
\[
\det\left[  \left[  \frac{\partial h_{i}}{\partial t_{j}}\left(  (a,b)\right)
\right]  _{_{\substack{1\leq i\leq l\\n-l^{\prime}+1\leq j\leq n}}}\left[
\frac{\partial h_{i}}{\partial z_{j}}\left(  (a,b)\right)  \right]
_{_{\substack{1\leq i\leq l\\1\leq j\leq l^{\prime\prime}}}}\right]  \neq0,
\]
with $l=l^{\prime}+l^{\prime\prime}$. Put $t^{\prime}=(t_{1},\ldots
,t_{n-l^{\prime}}),\ t^{\prime\prime}=(t_{n-l^{\prime}+1},\ldots
,t_{n}),\ z^{\prime}=(z_{1},\ldots,z_{l^{\prime\prime}})$ and $z^{\prime
\prime}=(z_{l^{\prime\prime}+1},\ldots,z_{n})$, hence $t=(t^{\prime}%
,t^{\prime\prime})$ and $z=(z^{\prime},z^{\prime\prime})$. Then the depending
variables are $t^{\prime\prime}$ and $z^{\prime}$. Consider a point
$(a,b)=(a^{\prime},a^{\prime\prime},b^{\prime},b^{\prime\prime})\in
(\mathbb{R}^{n-l^{\prime}}\times\mathbb{R}^{l^{\prime}}\times\mathbb{R}
^{l^{\prime\prime}}\times\mathbb{R}^{n-l^{\prime\prime}})\cap M$.  Then there
exist an open set $Q\subset\mathbb{R}^{2n-l}$, containing $(a^{\prime
},b^{\prime\prime})$, and a unique analytic function $\mathbf{q}
:Q\rightarrow\mathbb{R}^{l}$ such that $\mathbf{q}(a^{\prime},b^{\prime\prime
})=(a^{\prime\prime},b^{\prime})$ and $\mathbf{h}(t^{\prime},\mathbf{q}
(t^{\prime},z^{\prime\prime}),z^{\prime\prime})=0$, for any $(t^{\prime
},z^{\prime\prime})\in Q$. Set
\[
M_{(a,b)}:=
\]
\[
\left\{
\begin{array}
[c]{l}%
(t^{\prime},t^{\prime\prime},z^{\prime},z^{\prime\prime})\in W_{0}\times
W_{1}\mid(t^{\prime},z^{\prime\prime})\in Q,\,(t^{\prime\prime},z^{\prime}
)\in\mathbf{q}(Q)\text{ and }\\
t_{n-l^{\prime}+1}=q_{1}(t^{\prime},z^{\prime\prime}),\ldots,t_{n}
=q_{l^{\prime}}(t^{\prime},z^{\prime\prime}),z_{1}=q_{l^{\prime}+1}(t^{\prime
},z^{\prime\prime}),\ldots,z_{l^{\prime\prime}}=q_{l}(t^{\prime}
,z^{\prime\prime})
\end{array}
\right\} .
\]
Then
\[%
\begin{matrix}
\pi\mid_{M_{(a,b)}}: & M_{(a,b)} & \rightarrow & \mathbb{R}^{n}\\
& (t^{\prime},z^{\prime\prime}) & \mapsto & (t^{\prime},q_{1}(t^{\prime
},z^{\prime\prime}),\ldots,q_{l^{\prime}} (t^{\prime},z^{\prime\prime})),
\end{matrix}
\]
and the corresponding Jacobian matrix is
\[
J(t^{\prime},z^{\prime\prime})=\left[
\begin{matrix} 1&\dots&0&0&\dots&0\\ \hdotsfor{6}\\ 0&\dots&1&0&\dots&0\\ \frac{\partial q_1}{\partial t_{1}}&\dots&\frac{\partial q_1}{\partial t_{n-l'}}&\frac{\partial q_1}{\partial z_{l''+1}}&\dots&\frac{\partial q_{1}}{\partial z_{n}}\\ \hdotsfor{6}\\ \frac{\partial q_{l'}}{\partial t_{1}}&\dots&\frac{\partial q_{l'}}{\partial t_{n-l'}}&\frac{\partial q_{l'}}{\partial z_{l''+1}}&\dots&\frac{\partial q_{l'}}{\partial z_{n}} \end{matrix}\right]
(t^{\prime},z^{\prime\prime}).
\]
The point $(t^{\prime},t^{\prime\prime},z^{\prime},z^{\prime\prime})$ is a
critical point for $\pi\mid_{M_{(a,b)}}$ if $rank J(t^{\prime},z^{\prime
\prime}) <n$, that is,
\[
rank \left[  \frac{\partial q_{i}}{\partial z_{k}}(t^{\prime},z^{\prime\prime
})\right]  _{\substack{1\leq i\leq l^{\prime}\\k^{\prime}\leq k\leq
n}}<l^{\prime}=\min\{l^{\prime},n-l^{\prime\prime}\},
\]
where $k^{\prime}=l^{\prime\prime}+1$. As above, this is equivalent with $rank
\left[  \frac{\partial h_{i}}{\partial z_{k}}(t^{\prime},t^{\prime\prime
},z^{\prime},z^{\prime\prime})\right]  _{\substack{1\leq i\leq l^{\prime
}\\k^{\prime}\leq k\leq n}}<l^{\prime}$, hence $T_{l^{\prime},k^{\prime}}$ is
the set of critical values of $\pi\mid_{M_{(a,b)}}$.
\end{proof}

\bigskip Given $b\in(\mathbb{R}_{+}\setminus\{0\})^{n}$, we define
\[%
\begin{matrix}
T_{b}: & (\mathbb{R}^{\times})^{n} & \rightarrow & (\mathbb{R}^{\times})^{n}\\
& (x_{1},\ldots,x_{n}) & \mapsto & (b_{1}x_{1},\ldots,b_{n}x_{n}).
\end{matrix}
\]

\begin{lemma}
\label{compatibility} Let $\boldsymbol{f}$ be a strongly non-degenerate
analytic mapping at the origin with respect to $\Gamma(\boldsymbol{f})$. Let
$\mathcal{L}^{-1}(\mathcal{F}_{\boldsymbol{f}})$ be a decoupage induced by a
simplicial fan subordinated to $\Gamma(\boldsymbol{f})$. Then for almost every
$b\in(\mathbb{R}_{+}\setminus\{0\})^{n}$, the mapping $\boldsymbol{f}\circ
T_{b}$ is compatible with $\mathcal{L}^{-1}(\mathcal{F}_{\boldsymbol{f}})$.
\end{lemma}

\begin{proof}
We use all the notation introduced in Section \ref{Section comp}\ and
Definition \ref{def2}. Set $\overline{\Delta}=\left\{  \sum_{i=1}^{n}
\lambda_{i}a_{i}\mid\lambda_{i}\in\mathbb{R}_{+}\right\}  $ a cone in
$\mathcal{F}_{\boldsymbol{f}}$, $\{I,J,K\}$ a partition of $\{1,\ldots,n\}$
with $I\neq\emptyset$ and $J\neq\emptyset$, and $\Delta_{K}=\{\sum_{i\in
K}\lambda_{i}a_{i}\mid\lambda_{i}\in\mathbb{R}_{+},\ \lambda_{i}>0\}$ for
$K\subset\{1,\ldots,n\}$. Assume $\tau=\cap_{i\in I}F(a_{i})$ is non-empty and compact.

\textbf{Case (}$Card(K)\geq l$\textbf{).} Note that in this case necessarily
$l<n$ (since $Card(K)<n$). We set $\psi:=\mathcal{L}^{-1} \circ\theta$ where
\[%
\begin{array}
[c]{ccccc}%
\mathbb{R}^{n} &
\begin{array}
[c]{c}%
\theta\\
\longrightarrow
\end{array}
& \mathbb{R}^{n} &
\begin{array}
[c]{c}%
\mathcal{L}^{-1}\\
\longrightarrow
\end{array}
& \left(  \mathbb{R}_{+}\smallsetminus\left\{  0\right\}  \right)  ^{n},
\end{array}
\]
$\theta(z)=\sum_{i=1}^{n}z_{i}a_{i}$, and the extension of $\mathcal{L}$ to
$(\mathbb{R}_{+}\setminus\{0\})^{n}$ is denoted again by $\mathcal{L}$. We
also set
\[
U=\{z\in\mathbb{R}^{n}\mid z_{j}=0\ \text{if}\ j\notin K\text{,}
\,\text{and}\,z_{j}>0\ \text{if}\ j\in K\},
\]
and
\[
N=\left\{
\begin{array}
[c]{l}%
b\in(\mathbb{R}_{+}\setminus\{0\})^{n}\mid(\boldsymbol{f}_{\tau}\circ
T_{b}\circ\psi)(z)=0\quad\text{and }\\
rank \left[  \frac{\partial(f_{i,\tau}\circ T_{b}\circ\psi)}{\partial z_{k}%
}(z)\right] _{\substack{1\leq i\leq l\\k\in K}} < \min\{l, Card(K)\} =
l\text{,}\ \text{for some }z\in U
\end{array}
\right\}  .
\]
To establish this case, we show that $N$ has measure zero. We may assume that
$K=\left\{  k^{\prime},k^{\prime}+1,\ldots,n\right\}  $, by renaming the
coordinates, with $k^{\prime}\leq n-l+1$.

Note that for any $b\in(\mathbb{R}_{+}\setminus\{0\})^{n}$ there exists a
$t\in\mathbb{R}^{n}$ such that $b=\psi(t)$, and
\begin{align*}
(\boldsymbol{f}_{\tau}\circ T_{b}\circ\psi)(z)  &  =\boldsymbol{f}_{\tau
}(T_{\psi(t)}\circ\psi)(z)=\boldsymbol{f}_{\tau}(\psi_{1}(t)\psi_{1}
(z),\ldots,\psi_{n}(t)\psi_{n}(z))\\
&  =\boldsymbol{f}_{\tau}(\psi_{1}(t+z),\ldots,\psi_{n}(t+z))=(\boldsymbol{f}%
_{\tau}\circ\psi)(t+z).
\end{align*}
Since $\boldsymbol{f}$ is strongly non-degenerate at the origin with respect
to $\Gamma(\boldsymbol{f})$ and $\psi$ is an analytic isomorphism over
$(\mathbb{R}^{\times})^{n}$, by taking $l^{\prime}=l$ in Lemma \ref{sard} we
conclude that
\[
N^{\prime}=\left\{
\begin{array}
[c]{l}%
t\in\mathbb{R}^{n}\mid(\boldsymbol{f}_{\tau}\circ\psi)(t+z)=0\,\quad
\text{and}\\
rank \left[  \frac{\partial(f_{i,\tau}\circ\psi)}{\partial z_{k} }(t+z)\right]
_{\substack{1\leq i\leq l\\k^{\prime}\leq k\leq n}}<l\text{, for some }z\in U
\end{array}
\right\}
\]
has measure zero. Since $\psi$ is an analytic isomorphism, $N$ has measure
zero too.

\textbf{Case (}$0\leq Card(K)\leq l-1$\textbf{).} We show that the set of the
$b$'s such that $(\boldsymbol{f}_{\tau}\circ T_{b})(z_{0})=0$ for $z_{0}\in
S_{K}$ has measure zero. This is equivalent to show that $Y_{z_{0}}=\left\{  t
\in\mathbb{R}^{n} \mid(\boldsymbol{f}_{\tau}\circ\psi)(t+z_{0})=0\right\}  $,
for a $z_{0}\in U$, has measure zero. Finally $Y_{z_{0}}$ has measure zero,
because it is a proper $\mathbb{R}$-analytic subset of $\mathbb{R}^{n}$.
\end{proof}

\section{Short list of candidate poles}

\subsection{Integral restricted to a sector}

Let $a_{1},\ldots,a_{n}\in\mathbb{N}^{n}$, determining linearly independent
vectors of $\mathbb{R}^{n}$, and $\overline{\Delta}_{\tau}=\{\lambda_{1}a_{1}
+\cdots+\lambda_{n}a_{n}\mid\lambda_{i}\in\mathbb{R}_{+}\}$ a closed cone in a
fixed simplicial fan $\mathcal{F}_{\boldsymbol{f}}$ subordinated to
$\Gamma(\boldsymbol{f})$, such that $F(a)=\tau$ for any $a\in\Delta_{\tau}$.
Set
\[
I_{\overline{\Delta}_{\tau}}(s):=I_{\overline{\Delta}_{\tau}}(s,\boldsymbol{f}
,\Phi)=\int_{\mathcal{L}^{-1}(\overline{\Delta}_{\tau})}\Phi(x)|\boldsymbol{f}%
(x)|_{\mathbb{R}}^{s}\,\left\vert dx\right\vert ,
\]
for $\operatorname{Re}(s)>0$, where $\Phi$ is a smooth function with compact
support, contained in a neighborhood of the origin.

\subsubsection{The change of variables for a cone}

Set $x\in(0,1]^{n}$ such that $\mathcal{L}(x)\in\overline{\Delta}_{\tau}$.
Then
\[
\mathcal{L}(x)=(-\ln x_{1},\ldots,-\ln x_{n})=\lambda_{1}a_{1}+\cdots
+\lambda_{n}a_{n}=(-\ln y_{1})a_{1}+\cdots+(-\ln y_{n})a_{n},
\]
for some unique $y=(y_{1},\ldots,y_{n})\in(0,1]^{n}$ satisfying $\mathcal{L}
(y)=\lambda$. Thus
\[
\ln x_{i}=\sum_{j=1}^{n}a_{ij}(\ln y_{j})=\sum_{j=1}^{n}\ln(y_{j}^{a_{ij}
})=\ln\left(  \prod_{j=1}^{n}y_{j}^{a_{ij}}\right)  ,
\]
i.e.
\[
x_{i}=\prod_{j=1}^{n}y_{j}^{a_{ij}}\ \text{and }\ x^{\alpha}=\prod_{j=1}
^{n}y_{j}^{\langle a_{j},\alpha\rangle},
\]
where $a_{j}=(a_{ij})_{1\leq i \leq n}$. From these considerations, we define
\[%
\begin{matrix}
w: & [0,1]^{n} & \rightarrow & [0,1]^{n} & \\
& y & \mapsto & x=w(y), & \text{where}\ x_{i}=\prod_{j=1}^{n}y_{j}^{a_{ij}}.
\end{matrix}
\]
By using $w$ as a change of variables in $I_{\overline{\Delta}_{\tau}}(s)$,
one gets
\begin{equation}
I_{\overline{\Delta}_{\tau}}(s)= |det[a_{1},\ldots,a_{n}]|_{\mathbb{R}}
\int\limits_{[0,1]^{n} } \phi(y) \left(  \prod\limits_{j=1}^{n}y_{j}%
^{\sigma\left(  a_{j}\right)  -1}\right)  |\boldsymbol{f}(w(y))|_{\mathbb{R}%
}^{s}\,\left\vert dy\right\vert ,\label{Iphi equal Idelta}%
\end{equation}
where $det[a_{1},\ldots,a_{n}]$ denotes the determinant of the matrix with
columns $a_{1},\ldots,a_{n}$ and $\phi(y):=\Phi(w(y))$.

%Set
%\[
%I_{\Phi}(s):=\int\limits_{[0,1]^{n}}\left(  \prod\limits_{i=1}^{n}%
%y_{j}^{\sigma\left(  a_{j}\right)  -1}\right)  \Phi(y)|\boldsymbol{f}%
%(w(y))|_{\mathbb{R}}^{s}\,\left\vert dy\right\vert ,
%\]
%then
%\begin{equation}
%I_{\overline{\Delta}_{\tau}}(s)=det[a_{1},\ldots,a_{n}]I_{\Phi}(s),
%\label{Iphi equal Idelta}%
%\end{equation}
%since $[0,1]^{n}\setminus(0,1]^{n}$ has measure zero in $\mathbb{R}^{n}$.

\subsubsection{Description of $w^{-1}\{0\}$}

The set $w^{-1}\{0\}$ plays an important role in the considerations below. We
give here some properties of this set, proven in \cite[Lemma 5.1]{D-S} and
\cite[Lemma 5.2]{D-S}, that we will use later on.

\begin{lemma}
(1) Take $y\in\lbrack0,1]^{n}$. Set $I=\{i \in\{1,\dots, n\} \mid y_{i}=0\}$.
Assume that $I\neq\emptyset$ and that $\tau^{\prime}=\cap_{i\in I}F(a_{i}%
)\neq\emptyset$. Then
\[
y\in w^{-1}\{0\}\ \text{if and only if}\ \tau^{\prime}\ \text{is a compact
face of}\ \Gamma(\boldsymbol{f}).
\]

(2) If $V$ is a small enough neighborhood (resp. a neighborhood) of the origin
in $[0,1]^{n}$, then $w^{-1}(V)$ is a small enough neighborhood (resp.
neighborhood) of $w^{-1}\{0\}$ in $[0,1]^{n}$.
\end{lemma}

\subsubsection{Properties of $\boldsymbol{f}\circ w$}

%Assume that $I=\{ i \in \{1,\dots, n\} \mid  y_{i}=0\}\neq\emptyset$ and that
%$\tau=\cap_{i\in I}F(a_{i})\neq\emptyset$.
Writing $f_{i}=\sum_{m} c_{m,i} x^{m}$ we note for any $i=1,\ldots,l$ that
\begin{equation}
f_{i}(w(y))=\sum_{m}c_{m,i}(w(y))^{m}=\sum_{m}c_{m,i}\prod_{j=1}^{n}%
y_{j}^{\langle a_{j},m\rangle},\label{f(w)}%
\end{equation}
what we can write as
\begin{equation}
f_{i}(w(y))=\left(  \prod_{j=1}^{n}y_{j}^{d(a_{j})}\right)  {f}_{i}^{*}(y)
,\label{efe i de w}%
\end{equation}
where $d(a_{j})=\min_{x\in\Gamma\left(  \boldsymbol{f}\right)  }\langle
a_{j},x\rangle.$ If $supp(f_{i})\cap\tau\neq\emptyset$, then the minimum of
all $\langle a_{j},m\rangle$ is attained at $\tau$ and ${f}_{i}^{*}(0)\neq0$.
If $supp(f_{i})\cap\tau=\emptyset$, then ${f}_{i}^{*}(0)=0$.

From now on we pick a sufficiently small neighborhood $W$ of $w^{-1}\{0\}$ on
which the series in the right hand side of (\ref{f(w)}) converges for every
$i$. Then for every $i$ the restriction of this series to $W\cap[0,1]^{n}$
agrees with $f_{i} \circ w$.

\subsection{Meromorphic Continuation of $I_{\overline{\Delta}_{\tau}}(s)$}

\begin{theorem}
\label{Theorem4}  Let $\boldsymbol{f}:U (\subset\mathbb{R}^{n})
\longrightarrow\mathbb{R}^{l}$, with $\boldsymbol{f}\left(  0\right)  =0$, be
an analytic mapping, strongly non-degenerate at the origin with respect to
$\Gamma\left(  \boldsymbol{f}\right) $. Assume that $\tau=\cap_{i=1}%
^{n}F(a_{i})$ is a $0$-dimensional face of $\Gamma(\boldsymbol{f})$, and that
$\boldsymbol{f}$ is compatible with $\mathcal{L}^{-1}(\overline{\Delta}_{\tau
})$. There exists a neighborhood $V (\subset U)$ of the origin such that, if
$\Phi$ is a smooth function with support contained in $V$, the following
assertions hold.

\noindent(1) $I_{\overline{\Delta}_{\tau}}(s)$ is convergent and defines a
holomorphic function of $s$ on $\operatorname{Re}(s)>\max\{-l,-\frac
{\sigma\left(  a_{1}\right)  }{d(a_{1})},\dots,-\frac{\sigma\left(
a_{n}\right)  }{d(a_{n} )}\}$,

\noindent(2) $I_{\overline{\Delta}_{\tau}}(s)$ admits a meromorphic extension
to the complex plane with poles of order at most $n$. Furthermore, the poles
belong to
\[
\left(  \bigcup_{1\leq i\leq n}\left(  -\frac{\sigma\left(  a_{i}\right)
+\mathbb{N}}{d(a_{i})}\right)  \right)  \cup(-\left(  l+\mathbb{N}\right)  ).
\]

\noindent(3) Let $\rho$ be a positive integer and let $s_{0}$ be a candidate
pole of $I_{\overline{\Delta}_{\tau}}(s)$, with $s_{0}\notin-\left(
l+\mathbb{N}\right) $ (resp. $s_{0}\in-\left(  l+\mathbb{N}\right) $). A
necessary condition for $s_{0}$ to be a pole of $I_{\overline{\Delta}_{\tau}%
}(s)$ of order $\rho$, is that
\[
Card\left\{  i\mid s_{0}\in\left(  -\frac{\sigma\left(  a_{i}\right)
+\mathbb{N}}{d(a_{i})}\right)  \right\}  \geq\rho\ \text{(resp.}\ \geq
\rho-1\text{)} .
\]

\end{theorem}

\begin{proof}
We first prove (1) and (2). By (\ref{Iphi equal Idelta}) it is sufficient to
prove the result for
\begin{equation}
I_{\Omega}(s)=\int\limits_{[0,1]^{n}} \Omega(y) \left(  \prod\limits_{i=1}^{n}
y_{j}^{\sigma\left(  a_{j}\right)  -1}\right)  |\boldsymbol{f}
(w(y))|_{\mathbb{R}}^{s}\,\left\vert dy\right\vert ,\label{def i phi}%
\end{equation}
where $\Omega$ is a smooth function. By using a sufficiently fine partition of
the unity, one can express $I_{\Omega}(s)$ as a finite sum of analogous
integrals $I_{\Omega_{p}}(s)$, where $\Omega_{p}$ is a smooth function with
support contained in a small ball around a point $p$ belonging to the support
of $\Omega$. The relevant points $p$ to consider belong to $w^{-1}\{0\}$. In
the sequel we may and will assume that $supp(\Omega_{p})$ is as small as
necessary for the arguments that follow. Several cases occur.

\smallskip\textbf{Case 1 }($p$ is the origin of $\mathbb{R}^{n}$).\newline
Assume that $\Omega_{0}$ is a smooth function containing the origin.
%Since $\boldsymbol{f}%
%$ is strongly non-degenerate with respect to $\Gamma(\boldsymbol{f})$, then
There exists at least one index $i_{0}$ such that $supp(f_{i_{0}})\cap\tau
\neq\emptyset$. Hence, with the notation of (\ref{efe i de w}), there exists a
small neighborhood of the origin $V_{0,\tau}$ such that $\{0\}\subset
supp(\Omega_{0})\subset V_{0,\tau}$, and ${f}^{*}_{i_{0}}(y)\neq0$ for any
$y\in V_{0,\tau}$. Consequently, we have that $\sum_{i=1}^{l}({f}_{i}%
^{*}(y))^{2}>0$ on $V_{0,\tau}$, and then by Lemma \ref{Lema1} the integral
\begin{equation}
I_{\Omega_{0}}(s)=\int\limits_{[0,1]^{n}} \Omega_{0}(y) \left(  \prod
_{j=1}^{n} y_{j}^{sd(a_{j})+\sigma\left(  a_{j}\right)  -1}\right)  \left(
\sum_{i=1}^{l}({f}_{i}^{*}(y))^{2}\right)  ^{s/2}\,\left\vert dy\right\vert
\label{I1}%
\end{equation}
has a meromorphic continuation to $\mathbb{C}$, with poles (of order at most
$n$) belonging to
\[
\bigcup_{1\leq i\leq n}\left(  -\frac{\sigma\left(  a_{i}\right)  +\mathbb{N}
}{d(a_{i})}\right)  .
\]

\textbf{Case 2 }($p$ has exactly $r$ coordinates equal to zero, with $1\leq
r\leq n-1$).\newline After renaming the variables, we may suppose that the
first $r$ coordinates are zero, i.e. $p$ is the form $p=(0,\ldots,0,p_{r+1}
,\ldots,p_{n})$. Let $\tau^{\prime}$ be the first meet locus of the cone
$\Delta_{\tau^{\prime}}$, spanned by $a_{1},\ldots,a_{r}$; it is a compact
face of $\Gamma(\boldsymbol{f})$. Note that, similarly as in \cite[Chap. II,
\S \ 8, \ Lemme 9]{AVG},
\begin{equation}
f_{i}(w(y))=\prod_{j=1}^{r}y_{j}^{d(a_{j})}\left(  \widetilde{f}_{i}
(y_{r+1},\ldots,y_{n})+O_{i}(y_{1},\dots,y_{n})\right)  ,\label{efe i}%
\end{equation}
where the $\widetilde{f}_{i}$ are polynomials in $y_{r+1},\ldots,y_{n}$ and
the $O_{i}(y_{1},\ldots,y_{n})$ are analytic functions in $y_{1},\ldots,y_{n}$
but belonging to the ideal generated by $y_{1},\ldots,y_{r}$. Here
$\widetilde{f}_{i}$ is identically zero if and only if $supp(f_{i})\cap
\tau^{\prime}=\emptyset$. Furthermore,
\begin{equation}
f_{i,\tau^{\prime}}(w(y))=\prod_{j=1}^{r}y_{j}^{d(a_{j})}\left(  \widetilde
{f}_{i}(y_{r+1},\ldots,y_{n})\right)  .\label{efe i tau}%
\end{equation}
For the sequel we redefine $f^{*}_{i}(y)$ as
\[
f^{*}_{i}(y) = \widetilde{f}_{i} (y_{r+1},\ldots,y_{n})+O_{i}(y_{1}%
,\dots,y_{n}).
\]
Set $\widetilde{p}=(p_{r+1},\ldots,p_{n})\in(\mathbb{R}^{\times})^{n-r}$. To
accomplish the proof of Case 2, we need to study the following three subcases.

\smallskip\noindent\textbf{Subcase 2.1} (There exists an index $i$ such that
$\widetilde{f}_{i}(\widetilde{p})\neq0$). Then there exists a neighborhood
$V_{p,\tau^{\prime}}$ of $p=(0,\ldots,0,p_{r+1},\ldots,p_{n}),$ such that
${f}^{*}_{i}(y)\neq0$ for any $y\in V_{p,\tau^{\prime}}$ and such that
$supp(\Omega_{p}) \subset V_{p,\tau^{\prime}}$. Hence
\[
g(y_{1},\ldots,y_{n}):=\left(  \sum_{i=1}^{l}({f}^{*}_{i}(y_{1} ,\ldots
,y_{n}))^{2}\right)  ^{1/2}>0,
\]
for any $y\in V_{p,\tau^{\prime}}$, and, by Lemma \ref{Lema1},
\begin{equation}
I_{\Omega_{p}}(s)=\int\limits_{[0,1]^{n}} \Omega_{p}(y) \left(  \prod
_{j=1}^{r} y_{j}^{sd(a_{j})+\sigma\left(  a_{j}\right)  -1}\right)
g(y)^{s}\,\left\vert dy\right\vert \label{I2}%
\end{equation}
has a meromorphic continuation to the whole complex plain with poles contained
in
\begin{equation}
\bigcup_{1\leq i\leq r}\left(  -\frac{\sigma\left(  a_{j}\right)  +\mathbb{N}
}{d(a_{j})}\right)  .\label{poles r}%
\end{equation}
\textbf{Subcase 2.2} ($\widetilde{f}_{i}(\widetilde{p})=0$ for $i=1,\ldots,l$
and $\widetilde{p}\in(0,1)^{n-r}$). By the non-degeneracy condition, the fact
that $w$ is an analytic isomorphism on $(\mathbb{R}_{+} \setminus\{0\})^{n}$
and (\ref{efe i tau}), one gets for any $y=\left(  \left(  y_{1},\ldots
,y_{r}\right)  ,\widetilde{y}\right)  \in(\mathbb{R}_{+} \setminus\{0\})^{r}
\times(\mathbb{R}_{+} \setminus\{0\})^{n-r}\cap\{y\mid\widetilde{f}_{1}%
(\tilde{y})=\cdots=\widetilde{f}_{l}(\tilde{y})=0\}$ that
\[
rank \left[  \dfrac{\partial f_{i,\tau^{\prime}}}{\partial y_{j}
}(w(y))\right]  =
\]
\[
rank \left[
\begin{matrix}
0 & \dots & 0 & \prod\limits_{j=1}^{r}y_{j}^{d(a_{j})}\frac{\partial
\widetilde{f}_{1}}{\partial y_{r+1}}(\tilde{y}) & \dots & \prod\limits_{j=1}
^{r}y_{j}^{d(a_{j})}\frac{\partial\widetilde{f}_{1}}{\partial y_{n}}(\tilde
{y})\\
\vdots &  & \vdots & \vdots &  & \vdots\\
0 & \dots & 0 & \prod\limits_{j=1}^{r}y_{j}^{d(a_{j})}\frac{\partial
\widetilde{f}_{l}}{\partial y_{r+1}}(\tilde{y}) & \dots & \prod\limits_{j=1}
^{r}y_{j}^{d(a_{j})}\frac{\partial\widetilde{f}_{l}}{\partial y_{n}}(\tilde
{y})
\end{matrix}
\right]  =l.
\]
Now this implies for $\tilde{y}=(y_{r+1},\ldots,y_{n})\in(\mathbb{R}_{+}
\setminus\{0\})^{n-r} \cap\{y\mid\widetilde{f}_{1}(\tilde{y})=\cdots
=\widetilde{f}_{l}(\tilde{y})=0\}$ that
\[
rank \left[
\begin{matrix}
\frac{\partial\widetilde{f}_{1}}{\partial y_{r+1}} & \dots & \frac
{\partial\widetilde{f}_{1}}{\partial y_{n}}\\
\vdots & \dots & \vdots\\
\frac{\partial\widetilde{f}_{l}}{\partial y_{r+1}} & \dots & \frac
{\partial\widetilde{f}_{l}}{\partial y_{n}}%
\end{matrix}
\right]  (\widetilde{y}) =l .
\]
Then necessarily $l\leq n-r$, all $\widetilde{f}_{i}$ are non-zero
polynomials, and $rank \left[  \frac{\partial\widetilde{f}_{i}}{\partial
y_{j}}(\widetilde{p})\right]  =l$.

Since $\left[  \frac{\partial\widetilde{f}_{i}}{\partial y_{j}}(\widetilde
{p})\right]  =\left[  \frac{\partial(\widetilde{f}_{i}+O_{i})}{\partial y_{j}
}(p)\right]  $, cf. (\ref{efe i}), this last matrix having rank $l$ implies
that we can choose new coordinates $y^{\prime}=(y_{1},\ldots,y_{r}%
,y_{r+1}^{\prime},\ldots,y_{r+l}^{\prime},y_{r+l+1},\ldots,y_{n})$ in a
neighborhood $V_{p}$ of $p$, with $supp(\Omega_{p})\subset V_{p}$, such that
\[
f_{i}(\omega(y^{\prime}))= \prod_{j=1}^{r}(y_{j}^{\prime})^{d(a_{j})}
y^{\prime}_{r+i}
\]
for $i=1,\dots,l$ and hence
\[
|\boldsymbol{f}(w(y))|_{\mathbb{R}}^{s}=\left(  \sum_{i=1}^{l}f_{i}
^{2}(w(y^{\prime}))\right)  ^{s/2}=\prod_{j=1}^{r}(y_{j}^{\prime})^{sd(a_{j}
)}\left(  \sum_{i=1}^{l}(y_{r+i}^{\prime})^{2}\right)  ^{s/2}.
\]
Consequently
\begin{equation}
I_{\Omega_{p}}(s)=\int\widetilde{\Omega}_{p}(y^{\prime}) \prod_{j=1}^{r}%
(y_{j}^{\prime})^{sd(a_{j})+\sigma\left(  a_{j}\right)  -1}\left(  \sum
_{i=1}^{l}(y_{r+i}^{\prime})^{2}\right)  ^{s/2}\,|dy^{\prime}|,\label{I3}%
\end{equation}
where the integration is performed over $[0,1]^{r}\times U\times
\lbrack0,1]^{n-l-r}$, where $U$ is a small neighborhood of the origin in
$\mathbb{R}^{l}$. By applying Corollary \ref{Cor1}, we get a meromorphic
continuation to the whole complex plane for (\ref{I3}), with poles belonging
to
\[
\left(  \bigcup_{1\leq i\leq r}\left(  -\frac{\sigma\left(  a_{j}\right)
+\mathbb{N}}{d(a_{j})}\right)  \right)  \cup(-\left(  l+\mathbb{N}\right)  ).
\]

\noindent\textbf{Subcase 2.3} ($\widetilde{f}_{i}(\widetilde{p})=0$ for
$i=1,\ldots,l$ and $\widetilde{p}\in(0,1]^{n-r}$ with at least one coordinate
equal to $1$). By renaming the variables we may assume that $\widetilde{p}$
has the form $\widetilde{p}=(p_{r+1},\ldots,p_{r+t},1,\ldots,1)$, with
$p_{r+i}\in(0,1)$ for each $i=1,\ldots,t$.

\medskip\noindent\textbf{Claim.} \textsl{(1) If $0\leq t\leq l-1$, there
exists an index $i_{0}$ such that ${f}^{*}_{i_{0}}(p)\neq0$. }

\textsl{(2) If $t\geq l$, then
\[
rank \left[  p_{j}(1-p_{j})\frac{\partial{f}^{*}_{i} }{\partial y_{j}%
}(p)\right]  _{\substack{1\leq i\leq l\\1\leq j\leq n}}=l.
\]
}

\smallskip Assuming the claim, we conclude when $0\leq t\leq l-1$, by a
similar argument as the one given for Case 1, that $I_{\Omega_{p}}(s)$ has a
meromorphic continuation to $\mathbb{C}$ with poles contained in the set
(\ref{poles r}). When $t\geq l$, we can choose new coordinates $y^{\prime
}=(y_{1},\ldots,y_{r},y_{r+1}^{\prime},\ldots,y_{r+l}^{\prime},y_{r+l+1}
,\ldots,y_{n})$ in a neighborhood $V_{p}$ of $p$ such that $y_{r+i}^{\prime
}=f^{*}_{i}(y)$ for $i=1,\ldots,l$. In this coordinate system $I_{\Omega_{p}%
}(s)$ has the form (\ref{I3}), and thus meromorphic continuation is obtained
as in Subcase 2.2.

\medskip\noindent\textbf{Proof of the claim.} Set $I=\{i\mid\,p_{i}%
=0\}=\{1,\ldots,r\},\,J=\{i\mid\,p_{i}=1\}=\{r+t+1,\ldots,n\}$ and
$K=\{i\mid\,0<p_{i}<1\}=\{r+1,\ldots,r+t\}$. Set also
\[
A=\{y\in\lbrack0,1]^{n}\mid y_{j}=0\ \text{if}\ j\in I\text{, }\,y_{j}
=1\ \text{if}\ j\in J\text{, and}\,0<y_{j}<1\ \text{if}\ j\in K\},
\]
and
\[
\widehat{A}=\{z\in\lbrack0,1]^{n}\mid z_{j}=1\ \text{if}\ j\in I\cup J\text{,}
\ \text{and}\,0<z_{j}<1 \ \text{if}\ j\in K\}.
\]
Given $y\in A$, we denote by $\hat{y}$ the element of $\widehat{A}$ such that
$y_{k}=\hat{y}_{k}$ for $k\in K$. Set $\Delta_{K}=\{\sum_{j\in K}\lambda
_{j}a_{j}\mid\lambda_{j}\in\mathbb{R}_{+}\smallsetminus\left\{  0\right\}  \}$
and $S_{K}=\mathcal{L}^{-1}(\Delta_{K})$ as before.

If $0\leq t= Card(K)\leq l-1$, there exists an index $i_{0}$, such that
$\widetilde{f}_{i_{0}}(p)\neq0$, by the compatibility condition and
(\ref{efe i tau}). If $t=Card(K)\geq l$, the compatibility condition asserts,
for
\[
\widehat{y}= (1,\ldots,1,y_{r+1},\ldots,y_{r+t},1,\ldots,1)
\]
\[
\in\{\widehat{y} \in\widehat{A} \mid(f_{1,\tau^{\prime} }\circ w)(\widehat
{y})=\cdots=(f_{l,\tau^{\prime}}\circ w)(\widehat{y})=0\},
\]
that
\[
rank \left[  \frac{\partial(f_{i,\tau^{\prime}}\circ w)}{\partial y_{j}%
}(\widehat{y})\right]  _{\substack{1\leq i\leq l\\1\leq j\leq n}} = rank
\left[  \frac{\partial(f_{i,\tau^{\prime}}\circ w)}{\partial y_{j}}%
(\widehat{y})\right]  _{\substack{1\leq i\leq l\\j \in K}}  =l.
\]
Note for the first equality that the restriction of $\omega$ to $\widehat{A}$
is a parametrization of $S_{K}$. By (\ref{efe i}) and (\ref{efe i tau}) this
condition on $y$ is equivalent, for
\[
y=(0,\ldots,0,y_{r+1},\ldots,y_{r+t},1,\ldots,1)\in\{y \in A \mid f^{*}%
_{1}(y)=\cdots=f^{*}_{l}(y)=0\} ,
\]
to
\[
rank \left[  \tfrac{\partial f^{*}_{i}}{\partial y_{j}}(y)\right]
_{\substack{1\leq i\leq l\\j \in K}} = l ,
\]
what can be rewritten as

%\[
%rank_{\mathbb{R}}\left[  \tfrac{\partial(f_{i,\tau^{\prime}}\circ w)}{\partial
%y_{j}}(y)\right]  =rank_{\mathbb{R}}\left[
%\begin{matrix}
%0 & \dots & 0 & \frac{\partial\widetilde{f}_{1}}{\partial y_{r+1}}(\widetilde{y}) & \dots
%& \frac{\partial\widetilde{f}_{1}}{\partial y_{n}}(\widetilde{y})\\
%\vdots &  & \vdots & \vdots &  & \vdots\\
%0 & \dots & 0 & \frac{\partial\widetilde{f}_{l}}{\partial y_{r+1}}(\widetilde{y}) & \dots
%& \frac{\partial\widetilde{f}_{l}}{\partial y_{n}}(\widetilde{y})
%\end{matrix}
%\right]  =l,
%\]
%for $\widetilde{y}=(y_{r+1},\ldots,y_{r+t},1,\ldots,1)\in\{\widetilde{y}\mid\widetilde
%{f}_{1}(\widetilde{y})=\cdots=\widetilde{f}_{l}(\widetilde{y})=0\}$.
%Then, since $t\geq l$,
%\[
%rank_{\mathbb{R}}\left[
%\begin{matrix}
%\frac{\partial\widetilde{f}_{1}}{\partial y_{r+1}}(\tilde{y}) & \dots &
%\frac{\partial\widetilde{f}_{1}}{\partial y_{n}}(\tilde{y})\\
%\vdots &  & \vdots\\
%\frac{\partial\widetilde{f}_{l}}{\partial y_{r+1}}(\tilde{y}) & \dots &
%\frac{\partial\widetilde{f}_{l}}{\partial y_{n}}(\tilde{y})
%\end{matrix}
%\right]  =rank_{\mathbb{R}}\left[
%\begin{matrix}
%\frac{\partial\widetilde{f}_{1}}{\partial y_{r+1}}(\tilde{y}) & \dots &
%\frac{\partial\widetilde{f}_{1}}{\partial y_{r+k}}(\tilde{y})\\
%\vdots &  & \vdots\\
%\frac{\partial\widetilde{f}_{l}}{\partial y_{r+1}}(\tilde{y}) & \dots &
%\frac{\partial\widetilde{f}_{l}}{\partial y_{r+k}}(\tilde{y})
%\end{matrix}
%\right]  =l,
%\]
%for $\tilde{y}=(y_{r+1},\ldots,y_{r+k},1,\ldots,1)\in\{\widetilde{f}%
%_{1}(\tilde{y})=\cdots=\widetilde{f}_{l}(\tilde{y})=0\}$, and $l\leq n-r$.
%

\[
l= rank \left[  y_{j}(1-y_{j})\frac{\partial{f}^{*}_{i} }{\partial y_{j}%
}(y)\right] _{\substack{1\leq i\leq l\\1\leq j\leq n}}=
\]
\[
rank {\footnotesize \left[
\begin{matrix}
0 & \dots & 0 & y_{r+1}(1-y_{r+1})\frac{\partial{f}^{*}_{1}}{\partial y_{r+1}%
}({y}) & \dots & y_{r+t}(1-y_{r+t})\frac{\partial{f}^{*}_{1}}{\partial
y_{r+t}}({y}) & 0 & \dots & 0\\
\vdots &  & \vdots & \vdots &  & \vdots & \vdots &  & \vdots\\
0 & \dots & 0 & y_{r+1}(1-y_{r+1})\frac{\partial{f}^{*}_{l}}{\partial y_{r+1}%
}({y}) & \dots & y_{r+t}(1-y_{r+t})\frac{\partial{f}^{*}_{l}}{\partial
y_{r+t}}({y}) & 0 & \dots & 0
\end{matrix}
\right]  .}
\]
This finishes the proof of the claim.

\smallskip The third part of the theorem follow from (\ref{I1}), (\ref{I2})
and (\ref{I3}) by applying Lemma \ref{Lema1} and Corollary \ref{Cor1}.
\end{proof}

\medskip Finally we state in the following theorem the meromorphic
continuation of
\[
I(s) = I(s,\boldsymbol{f},\Phi)=\int_{\mathbb{R}_{+}^{n} }\Phi(x)\left\vert
\boldsymbol{f}(x)\right\vert _{\mathbb{R}}^{s}\,\left\vert dx\right\vert .
\]
Recall the notation
\[
\mathcal{P}(\xi) =\left\{  -\frac{\sigma\left(  \xi\right)  +k}{d\left(
\xi\right)  } \mid k\in\mathbb{N} \right\}  .
\]
for $\xi\in\mathfrak{D} (\Gamma\left(  \boldsymbol{f}\right)  )$.

\begin{theorem}
\label{Theorem5} Let $\boldsymbol{f}:U (\subset\mathbb{R}^{n}) \longrightarrow
\mathbb{R}^{l}$, with $\boldsymbol{f}\left(  0\right)  =0$, be an analytic
mapping, strongly non-degenerate at the origin with respect to $\Gamma\left(
\boldsymbol{f}\right) $. There exists a neighborhood $V (\subset U)$ of the
origin such that, if $\Phi$ is a smooth function with support contained in
$V$, the following assertions hold:

\noindent(1) $I(s)$ converges and defines a holomorphic function on
$\operatorname{Re}(s)>\max\left\{  -\gamma_{0}\left(  \boldsymbol{f}\right)
,-l\right\}  $, with
\[
\gamma_{0}\left(  \boldsymbol{f}\right)  =\min_{\xi\in\mathfrak{D}
(\Gamma\left(  \boldsymbol{f}\right)  )}\left\{  \frac{\sigma\left(
\xi\right)  }{d\left(  \xi\right)  }\right\}  ,
\]

\noindent(2) $I(s)$ admits a meromorphic continuation to the whole complex
plane with poles of order at most $n$, and the poles belong to
\[
\left(  \bigcup_{\xi\in\mathfrak{D} (\Gamma(\boldsymbol{f}) )} \mathcal{P}
(\xi) \right)  \cup(-\left(  l+\mathbb{N}\right)  ).
\]

\noindent(3) Let $\rho$ be a positive integer and let $s_{0}$ be a candidate
pole of $I(s)$, with $s_{0}\notin-\left(  l+\mathbb{N}\right) $ (resp.
$s_{0}\in-\left(  l+\mathbb{N}\right)  $). A necessary condition for $s_{0}$
to be a pole of $I(s)$ or order $\rho$, is that there exists a face $\tau$ of
$\Gamma\left(  \boldsymbol{f}\right) $, of codimension $\rho$ (resp. $\rho
-1$), such that $s_{0} \in\mathcal{P}(\xi)$ for all facets of $\Gamma\left(
\boldsymbol{f}\right) $ containing $\tau$.
\end{theorem}

\begin{proof}
Pick a simplicial fan $\mathcal{F}_{\boldsymbol{f}}$ subordinated to
$\Gamma(\boldsymbol{f})$. By Lemma \ref{compatibility} there exists
$b\in(\mathbb{R}_{+}\setminus\{0\})^{n}$ such that $\boldsymbol{f}\circ T_{b}$
is compatible with $\mathcal{F}_{\boldsymbol{f}}$. Since the change of
variables $x=T_{b}(y)$ does not affect the meromorphic continuation of $I(s)$,
we may assume that $\boldsymbol{f}$ is compatible with the decoupage
$\mathcal{L}^{-1}(\mathcal{F}_{\boldsymbol{f}})$. Parts (1) and (2) follow
immediately from Theorem \ref{Theorem4} by the decomposition
\[
I(s)=\int_{\mathbb{R}_{+}^{n}}\Phi(x)\left\vert \boldsymbol{f}(x)\right\vert
_{\mathbb{R}}^{s}\,dx=\sum_{\overline{\Delta}_{\tau}\in\mathcal{F}
}I_{\overline{\Delta}_{\tau}}(s).
\]
Part (3) also follows from Theorem \ref{Theorem4} by exactly the same argument
as in the proof of \cite[Th\'eor\`eme 6.1(iii)]{D-S}, using \cite[Corollaire
2.2]{D-S}.
\end{proof}

%\begin{remark}
%\label{note7}If we only assume that $\boldsymbol{f}$ is strongly
%non-degenerate at the origin with respect to $\Gamma(\boldsymbol{f})$ (see
%Definition \ref{def1}), then Theorems \ref{Theorem4} and \ref{Theorem5} are
%valid in a small neighborhood around the origin.
%\end{remark}

\end{document}